\documentclass[11pt]{amsart}
\usepackage{amssymb,mathrsfs,graphicx,enumerate,color}
\usepackage{colortbl}
\definecolor{black}{rgb}{0.0, 0.0, 0.0}
\definecolor{red}{rgb}{1.0, 0.5, 0.5}

\title[   ]{$L^2$-type contraction for shocks of scalar viscous conservation laws with strictly convex flux}

\author[Kang]{Moon-Jin Kang}
\address[Moon-Jin Kang]{
\newline Department of Mathematic \& Research Institute of Natural Sciences, \newline Sookmyung Women's University, Seoul 140-742, Korea}
\email{moonjinkang@sookmyung.ac.kr}

\newtheorem{theorem}{Theorem}[section]
\newtheorem{lemma}{Lemma}[section]
\newtheorem{corollary}{Corollary}[section]
\newtheorem{proposition}{Proposition}[section]
\newtheorem{remark}{Remark}[section]

\newcommand{\bbr}{\mathbb R}

\newcommand{\deo}{\delta_0}

\numberwithin{figure}{section}
\newcommand{\beq}{\begin{equation}}
\newcommand{\eeq}{\end{equation}}
\newcommand{\bsp}{\begin{split}}
\newcommand{\esp}{\end{split}}

\newcommand{\intr}{{\int_{\RR_-}}}
\newcommand{\lpa}{\left(}
\newcommand{\rpa}{\right)}

\newcommand{\RR}{{\mathbb R}}

\newcommand{\s}{\sigma}

\def\eps{\varepsilon }

\newcommand\adots{\mathinner{\mkern2mu\raise1pt\hbox{.}
\mkern3mu\raise4pt\hbox{.}\mkern1mu\raise7pt\hbox{.}}}

\def\charf {\mbox{{\text 1}\kern-.30em {\text l}}}

\begin{document}
\bibliographystyle{plain}

\date{\today}


\thanks{\textbf{Acknowledgment.}  The author was partially supported by the NRF-2017R1C1B5076510.
}

\begin{abstract}
We study the $L^2$-type contraction property of large perturbations around shock waves of scalar viscous conservation laws with strictly convex fluxes in one space dimension.
The contraction holds up to a shift, and it is measured by a weighted related entropy, for which we choose an appropriate entropy associated with the strictly convex flux. In particular, we handle shocks with small amplitude. This result improves the recent article \cite{Kang-V-1} of the author and Vasseur on $L^2$-contraction property of shocks to scalar viscous conservation laws with a special flux, that is almost the Burgers flux.
\end{abstract}

\maketitle \centerline{\date}

\tableofcontents

\section{Introduction}
\setcounter{equation}{0}

A scalar conservation law in one space dimension is a first order partial differential equation
of the form
\[
u_t + f(u)_x=0,
\]
where $u$ represents a conserved quantity, and $f$ is the flux function describing certain physical phenomenon. This scalar equation has been extensively used in macroscopic modeling of various phenomena such as vehicular traffic flow, pedestrian flow, driven thin film flow, and so on (see for example \cite{AF,BDC,BD,Bertozzi,BMFC,CMW,DG,GK,LW}). 
An important feature of this equation is an emergence of shock waves as severe singularities even for smooth initial data. In the modeling of traffic flow, the traffic jam is regarded as the shock wave. For a variety of purposes, a linear diffusion is added in the above equation, in other words a scalar viscous conservation law is considered. The scalar viscous conservation law is an important object in studying for not only the viscous shock layer but also the conservation law with Brownian motion.

We consider a scalar viscous conservation law with a strictly convex flux $f$ in one space dimension: 
\begin{align}
\begin{aligned} \label{main}
&u_t + f(u)_x = u_{xx}, \quad t>0,~x\in \bbr,\\
&u(0,x) = u_0(x).
\end{aligned}
\end{align}

For any strictly convex flux $f$, and any real numbers $u_{-}, u_{+}$  with $u_{-}>u_{+}$, the scalar conservation law \eqref{main} admits viscous shock waves $S$ connecting the two end states  $u_-$ and $u_+$. More precisely, there exists a smooth traveling wave $S(x-\sigma t)$ as a solution of 
\begin{align}
\begin{aligned}\label{shock} 
& -\sigma S^{\prime} + f(S)^{\prime} = S^{\prime\prime},\\
& \lim_{\xi\to \pm\infty} S(\xi) = u_{\pm},
\end{aligned}
\end{align} 
where $\sigma$ is the velocity of $S$ determined by the Rankine-Hugoniot condition:
\beq\label{RH-con}
\sigma=\frac{f(u_-)-f(u_+)}{u_--u_+}.
\eeq
Indeed, it follows from \eqref{shock}-\eqref{RH-con} that there exists $S$ such that $u_+<S<u_-$ and
\begin{align}
\begin{aligned}\label{shock-1} 
S^{\prime} &= - \sigma (S -u_\pm) + f(S) -f(u_\pm)\\
& = (S-u_\pm) \Big( \frac{f(S)-f(u_\pm)}{S-u_\pm} -\frac{f(u_-)-f(u_+)}{u_--u_+} \Big) <0,
\end{aligned}
\end{align}
where the strict convexity of $f$ implies the last inequality. However, when the flux $f$ periodically depends on the space variable, the existence of shocks is nontrivial issue (see \cite{Dalibard,D-K}). 

In this article, we aim to show the contraction property of any large perturbations of shock waves to \eqref{main} with strictly convex fluxes. The contraction holds up to a shift, and it is measured by a weighted relative entropy. 
Our result improves the recent article \cite{Kang-V-1} of the author and Vasseur about $L^2$-contraction property of shocks to \eqref{main} with special fluxes as small perturbations of the quadratic Burgers flux, in the sense that
$f(u)=au^2 + g(u)$ for $a>0$ and $g$ satisfying $\|g^{\prime\prime}\|_{L^{\infty}(\bbr)} <\frac{2}{11}a$. 
Notice that the authors in \cite[Theorem 1.2]{Kang-V-1} shows that for any shock, there exists a strictly convex  flux and a smooth initial datum such that there is no $L^2$-contraction property.\\
We will get the contraction property even for general strictly convex fluxes. For that, we employ a weighted relative entropy to measure the contraction instead of using $L^2$-distance as in \cite{Kang-V-1}. 
Our method is based on the new approach introduced by \cite{Kang-V-NS17}, in which they proved the contraction property of shocks to the barotropic Navier-Stokes system. The new method of \cite{Kang-V-NS17} was also used in studies on contraction of traveling waves  \cite{CKKV,Kang-V_unique}.  We here handle shocks with small amplitude as in \cite{Kang-V-NS17}.
However, compared to \cite{Kang-V-NS17}, our main difficulty is to handle any strictly convex flux. To overcome it, we take advantage of the fact that any function can be an entropy of scalar conservation laws \eqref{main}. More precisely, for a given strictly convex flux, we will choose a strictly convex entropy associated to the flux, and consider  the relative entropy defined by the entropy. \\

Notice that since the relative entropy is locally quadratic, the contraction measured by the relative entropy can be regarded as $L^2$-type contraction. We need to use only one entropy to get the $L^2$-type contraction measured by the relative entropy, contrary to the $L^1$-contraction by Kruzkhov \cite{K1}, in which the large family of entropies $\eta_k(u):=|u-k|, ~k\in\bbr$ are used. However, since many systems of conservation laws (including Euler systems) have only one nontrivial entropy, using only one entropy for the contraction property is a remarkable program. \\
In this sense, this article follows the same context as in the works \cite{CV,Kang,Kang-V-NS17,Kang-V_unique,KVARMA,Kang-V-1,KVW,Krupa-V,Leger,Serre-Vasseur,SV_16,SV_16dcds,Vasseur-2013,VW} of Vasseur and his collaborators about stability (or contraction, uniqueness) of viscous (or inviscid) shocks by the relative entropy method.
 We also refer to Dafermos \cite{Dafermos1} and DiPerna \cite{DiPerna} for studies on uniqueness and stability of Lipschitz solution to conservation laws, which were obtained by using the relative entropy method at the first time.

\subsection{Main result}
By the theory of conservation law, any convex function $\eta$ is an entropy of the scalar conservation law \eqref{main} since there exists an entropy flux $q$ as
\beq\label{def-q}
q(u)=\int^u \eta'(w) f'(w) dw.
\eeq
Therefore, we can choose any entropy we need. We will consider an appropriate entropy associated with the strictly convex flux $f$ of \eqref{main} as the following hypotheses.

$\bullet$ {\bf Hypotheses:} Let $f : \bbr\to\bbr$ be a smooth (or ${C}^3$) and strictly convex function satisfying the growth condition: 
\beq\label{growth}
\mbox{there exist constants $a, b>0$ such that $|f(u)|\le  a e^{b|u|}$ for all $u\in\bbr$.}
\eeq
Then, there exists a strictly convex entropy $\eta$ such that $\eta$ is a strictly convex smooth function satisfying the following hypotheses:
\begin{itemize}
\item  
(${\mathcal H}$1): There exists a constant $\alpha>0$ such that $\eta''(u)\ge \alpha$ and $\eta''''(u)\ge \alpha$ for all $u\in\bbr$.\\

\item  
(${\mathcal H}$2): For a given constant $\theta>0$, there exists a constant $C>0$ such that 
for any $u,v\in \bbr$ with $|v|\le \theta$, the following inequalities holds:\\

(i) $ |\eta'(u)-\eta'(v)|^2 {\mathbf 1}_{\{ |u| \le 2\theta \}} + |\eta'(u)-\eta'(v)|{\mathbf 1}_{\{|u| >2\theta \}} \le C \eta(u|v)$\\

(ii) $|f(u)-f(v)|\le C |\eta'(u)-\eta'(v)|$\\

(iii) $|\eta''(u)-\eta''(v)|\le C |\eta'(u)-\eta'(v)|$\\

(iv) $\Big|\int^u_v \eta''(w)f(w)\, dw \Big| \le C \Big( |\eta'(u)-\eta'(v)| {\mathbf 1}_{\{ |u| \le 2\theta \}}  +|\eta'(u)-\eta'(v)|^2 {\mathbf 1}_{\{|u| >2\theta \}}\Big)$.
\end{itemize}

 \vskip1cm
 
In Appendix \ref{app-entropy}, we present the proof of the existence of entropies verifying the above hypotheses.
Under an entropy $\eta$ satisfying the above hypotheses, we consider the relative entropy
\[
\eta(u|v):= \eta(u)-\eta(v)-\eta'(v)(u-v).
\]
Note that since the entropy $u\mapsto\eta(u)$ is strictly convex, the relative entropy $\eta(\cdot | \cdot)$ is positive definite, that is, $\eta(u_1 | u_2)\geq 0$ for any $u_1$ and $u_2$, and 
$\eta(u_1 | u_2)= 0$ if and only if $u_1=u_2$.\\ 
The following result provides the contraction property of large perturbations measured by the relative entropy with some weight.

\begin{theorem}\label{thm_general}
Consider an equation \eqref{main} with any smooth (or ${C}^3$) strictly convex flux $f$ satisfying the growth condition \eqref{growth}. Let $\eta$ be an entropy of \eqref{main} satisfying the hypotheses (${\mathcal H}1$)-(${\mathcal H}2$).\\
For any $u_-\in \bbr$ with $f''(u_-)>0$, there exists  a constant $\delta_0\in(0,1)$ such that the following holds.\\
For any $\eps, \lambda>0$ with $\delta_0^{-1}\eps<\lambda<\delta_0$, let $u_+:=u_--\eps$. Then there exists  a smooth monotone function $a:\bbr\to\bbr^+$ with $\lim_{x\to\pm\infty} a(x)=1+a_{\pm}$ for some  constants $a_-, a_+$ with $|a_+-a_-|=\lambda$ such that the following holds.\\
Let $S$ be the viscous shock connecting $u_-$ and $u_+$ as a solution of \eqref{shock}.
For any $T>0$, and any initial data $u_0\in L^\infty(\bbr)$ satisfying $\int_{-\infty}^{\infty} \eta(u_0| S) dx<\infty$, the equation \eqref{main} admits a bounded weak solution $u$ such that
\beq\label{cont_main}
\int_{-\infty}^{\infty} a(x) \eta\big(u(t,x+X(t))|S (x)\big) dx \le \int_{-\infty}^{\infty} a(x) \eta\big(u_0(x)|S (x)\big) dx, \quad t\le T,
\eeq
Here, there exists a unique absolutely continuous shift $X$  such that $X\in W^{1,1}_{loc}((0,T))$ and
\begin{align}
\begin{aligned} \label{est-shift}
&|\dot X(t)|\le \frac{1}{\eps^2}(1 + h(t)),\quad t\le T,\\
&\mbox{for some positive function $h$ satisfying}\quad\|h\|_{L^1(0,T)} \le\frac{2\lambda}{\delta_0\eps}\int_{-\infty}^{\infty} \eta(u_0| S) dx.
\end{aligned}
\end{align}
\end{theorem}


\begin{remark}
The contraction \eqref{cont_main} implies the uniform $L^2$-stability (up to shift) in the following sense: there exists a positive constant $C=C(\alpha, \delta_0, \|u_0\|_{L^\infty(\bbr)})$ such that
\beq\label{L2stab}
\int_{-\infty}^{\infty} |u(t,x+X(t))-S (x)|^2 dx \le C \int_{-\infty}^{\infty} | u_0(x)-S (x)|^2 dx.
\eeq
Indeed, since
 $\|u\|_{L^\infty(\bbr)} \le \|u_0\|_{L^\infty(\bbr)}$ by the maximum principle, it follows from \eqref{e-est0} and \eqref{releq} that the relative entropy is comparable with the $L^2$-distance. Moreover, since $\|a-1\|_{L^\infty(\bbr)}\le\lambda<\delta_0<1$, the contraction \eqref{cont_main} implies the above $L^2$-stability.
 \end{remark}

\begin{remark}
From the condition $\delta_0^{-1}\eps<\lambda<\delta_0$ of Theorem \ref{thm_general}, we have smallness conditions on three quantities $\eps, \eps/\lambda$ and $\lambda$, since $\eps<\delta_0^2, ~\eps/\lambda<\delta_0$.   
Thus the contraction \eqref{cont_main} holds for shocks with small amplitude $\eps$. 
However this result improves \cite[Theorem 1.1]{Kang-V-1} in the following sense. Theorem \ref{thm_general} deals with any strictly convex fluxes satisfying the exponential growth condition \eqref{growth} while \cite[Theorem 1.1]{Kang-V-1} covered only the approximated Burgers fluxes.
We also mention the contraction \eqref{cont_main} is measured by the relative entropy with the weight $a$, while the contraction of \cite[Theorem 1.1]{Kang-V-1} is measured by $L^2$-distance.
\end{remark}

 \vskip0.3cm

In fact, the growth condition \eqref{growth} and the hypotheses (${\mathcal H}$2) can be relaxed in any bounded interval $[-M,M]$, as follows.\\
Let $f$ be a smooth (or ${C}^3$) strictly convex flux for \eqref{main}. Then, there exists a strictly convex entropy $\eta$ such that $\eta$ satisfying the hypotheses (${\mathcal H}1$) and the relaxed hypotheses of (${\mathcal H}2$):
\begin{itemize}
\item  
(${\mathcal H}2'$): For a given constant $M>0$, there exists a constant $C>0$ such that 
for any $u,v\in \bbr$ satisfying $ |u| \le M $ and $|v|\le M/2$, the following inequalities holds:\\

(i) $ |\eta'(u)-\eta'(v)|^2  \le C \eta(u|v)$\\

(ii) $|f(u)-f(v)|\le C |\eta'(u)-\eta'(v)|$\\

(iii) $|\eta''(u)-\eta''(v)|\le C |\eta'(u)-\eta'(v)|$\\

(iv) $\Big|\int^u_v \eta''(w)f(w)\, dw \Big| \le C |\eta'(u)-\eta'(v)| $.
\end{itemize}

In Appendix \ref{app-entropy}, we present the proof of the existence of entropies verifying the above hypotheses (${\mathcal H}1$) and (${\mathcal H}2'$).
Under an entropy satisfying the above hypotheses, Theorem \ref{thm_general} implies the following Corollary.

\begin{corollary}\label{coro}
Consider an equation \eqref{main} with any smooth (or ${C}^3$) strictly convex flux $f$.
For given constants $u_-\in \bbr$ and $M>0$ satisfying $f''(u_-)>0$ and $M\ge 6|u_-|$,
let $\eta$ be an entropy of \eqref{main} satisfying the hypotheses (${\mathcal H}1$) and (${\mathcal H}2'$) with the constant $M$.\\
Then, there exists  a constant $\delta_0\in(0,1)$ such that the following holds.\\
For any $\eps, \lambda>0$ with $\delta_0^{-1}\eps<\lambda<\delta_0$, let $u_+:=u_--\eps$. Then there exists  a smooth monotone function $a:\bbr\to\bbr^+$ with $\lim_{x\to\pm\infty} a(x)=1+a_{\pm}$ for some  constants $a_-, a_+$ with $|a_+-a_-|=\lambda$ such that the following holds.\\
Let $S$ be the viscous shock connecting $u_-$ and $u_+$ as a solution of \eqref{shock}.
For any $T>0$, and any initial data satisfying $\|u_0\|_{L^\infty(\bbr)} \le M$ and $\int_{-\infty}^{\infty} \eta(u_0| S) dx<\infty$, the equation \eqref{main} admits a bounded weak solution $u$ such that
\eqref{cont_main} and \eqref{est-shift} hold true.
\end{corollary}

\begin{remark}
Corollary \ref{coro} improves \cite[Theorem 1.2]{Kang-V-1} in the following sense. Since it follows from \cite[Theorem 1.2]{Kang-V-1} that there exist a strictly convex flux $f$ and smooth initial perturbation $u_0-S$ (thus, $u_0\in L^\infty$) such that the $L^2$-contraction property fails even for small shocks. 
However, Corollary \ref{coro} provides the contraction property for any strictly convex flux and any bounded initial perturbations.
\end{remark}

\begin{remark}
Contrary to Theorem \ref{thm_general}, the constant $\delta_0$ of Corollary \ref{coro} depends on $M$.
Therefore, in Corollary \ref{coro}, choice of $M$ may affect the range of choices of shock strength $\eps$, and of $L^\infty$-norm of $u_0$, instead the growth condition \eqref{growth} is not needed anymore. 
\end{remark}

\begin{remark}\label{rem-s}
 Notice that it is enough to prove Theorem \ref{thm_general} (and Corollary \ref{coro}) only for shocks moving at positive velocity, i.e., $\s>0$ in \eqref{RH-con}.
Indeed, if $\s<0$ or $\s=0$, then we consider  \eqref{main} with a strictly convex flux $g(u)=-2\s u+f(u)$ or $g(u)=\s u +f(u)$ respectively instead of the original flux $f$. 
Therefore from now on, we assume $\s>0$.\\
\end{remark}

 \vskip0.3cm

The rest of the paper is organized as follows. In Section 2, we present useful lemmas for the proof of main result. The Section 3 and 4 are devoted to the proof of Theorem \ref{thm_general}. 

 \vskip0.3cm

$\bullet$ {\bf Notations :}
Throughout the paper, $C$ denotes a positive constant which may change from line to line, but which stays independent on $\eps$ (the shock strength) and $\lambda$ (the total variation of the function $a$). 
The paper will consider two smallness conditions, one on $\eps$, and the other on $\eps/\lambda$.

\section{Preliminaries}\label{section_pre}
\setcounter{equation}{0}
In this section, we present properties of small shocks and the relative entropy, and a key representation by the relative entropy method.

\subsection{Small shock waves}
In this subsection, we present  useful properties of the shock waves $S$ with small amplitude $u_--u_+=\eps$. 

In the sequel, without loss of generality, we consider the shock $S$ satisfying $S(0)=\frac{u_-+u_+}{2}$. 

\begin{lemma}
For any $n_-\in\bbr$ with $f''(u_-)>0$,  there exists $\eps_0>0$ such that for any $0<\eps<\eps_0$ the following is true. 
Let $S$ be the shock wave connecting end states $u_-$ and $u_+:=u_--\eps$ such that $S(0)=\frac{u_-+u_+}{2}$. Then, there exist constants $C, C_1, C_2>0$ such that
\beq\label{tail}
-C^{-1}\eps^2 e^{-C_1 \eps |\xi|} \le S'(\xi) \le -C\eps^2 e^{-C_2 \eps |\xi|},\quad \forall\xi\in\bbr,
\eeq
as a consequence,
\beq\label{low-S}
\inf_{\left[-\frac{1}{\eps},\frac{1}{\eps}\right]}| S' |\ge C\eps^2.
\eeq
Moreover,
\beq\label{sec-S}
|S''(\xi)|\le C\eps |S'(\xi)|,\quad \forall\xi\in\bbr.
\eeq

\end{lemma}
\begin{proof} 
$\bullet$ {\it proof of \eqref{tail}-\eqref{low-S}} :
We recall \eqref{shock-1} as
\[
S' = (S-u_+) \Big( \frac{f(S)-f(u_+)}{S-u_+} -\frac{f(u_-)-f(u_+)}{u_--u_+} \Big).
\]
Then we have
\[
\frac{S'}{S-u_+}= \frac{f(S)-f(u_+)}{S-u_+} -\frac{f(u_-)-f(u_+)}{u_--u_+} .
\]
Let  $\varphi:\bbr^+\to\bbr$ be a smooth function defined by 
\[
\varphi(u):=\frac{f(u)-f(u_+)}{u-u_+}.
\]
Then, the above equality can be written as
\beq\label{temm}
\frac{S'}{S-u_+}= \varphi(S)-\varphi(u_-).
\eeq
Since $0<u_--S<\eps<\eps_0$, taking $\eps_0$ small enough, there exists a constant $C>0$ (depending only on $u_-$) such that
\beq\label{apply-v}
|\varphi(S)-\varphi(u_-)-\varphi'(u_-)(S-u_-)|\le C(S-u_-)^2 \le C\eps_0 (u_- -S).
\eeq
Notice that since
\[
\varphi'(u_-)=\frac{f'(u_-)(u_--u_+)-\big(f(u_-)-f(u_+)\big)}{(u_--u_+)^2}=\frac{f''(u_*)}{2}\quad\mbox{for some }~ u_*\in(u_+,u_-),
\]
using the continuity of $f''$ and taking $\eps_0$ small enough, we have
\[
f''(u_-) \ge \varphi'(u_-) \ge \frac{f''(u_-)}{4}>0\quad\mbox{(by }~f''(u_-)>0 ).
\]
That is, $\varphi'(u_-)$ is bounded from below and above uniformly in $\eps$.\\
Therefore, for $\eps_0$ small enough, we have
\[
\frac{f''(u_-)}{8} (u_- -S) \le  \varphi(u_-) -\varphi(S) \le 2f''(u_-) (u_- -S).
\]
Then, it follows from \eqref{temm} that
\beq\label{gein}
2f''(u_-) (S -u_-)(S-u_+)  \le S' \le\frac{f''(u_-)}{8}  (S-u_-) (S-u_+) .
\eeq
To prove the estimate \eqref{tail}, we first observe that since $S'<0$ and $S(0)=\frac{u_-+u_+}{2}$, we have
\begin{align}
\begin{aligned}\label{bothcase}
&\xi\le0 ~\Rightarrow~u_--u_+ \ge S(\xi)-u_+\ge S(0)-u_+=\frac{u_--u_+}{2},\\
&\xi\ge0 ~\Rightarrow~u_--u_+ \ge u_--S(\xi)\ge u_-- S(0)=\frac{u_--u_+}{2}.
\end{aligned}
\end{align} 
Then, using \eqref{gein} and \eqref{bothcase} with $u_--u_+=\eps$, we have 
\begin{align*}
\begin{aligned}
&\xi\le0 ~\Rightarrow~-C^{-1}\eps (u_- - S)\le S' \le -C\eps (u_- - S),\\
&\xi\ge0 ~\Rightarrow~-C^{-1}\eps (S -u_+)\le S' \le -C\eps (S -u_+).
\end{aligned}
\end{align*} 
Thus,
\begin{align*}
\begin{aligned}
&\xi\le0 ~\Rightarrow~-C^{-1}\eps (u_- -S)\ge (u_- -S)' \ge -C\eps (u_- - S),\\
&\xi\ge0 ~\Rightarrow~-C^{-1}\eps (S -u_+) \le (S-u_+)' \le -C\eps (S -u_+).
\end{aligned}
\end{align*} 
These together with $S(0)=\frac{u_-+u_+}{2}$ imply
\begin{align*}
\begin{aligned}
&\xi\le0 ~\Rightarrow~C^{-1}\eps e^{-C_2\eps|\xi|}\le u_- -S \le C\eps e^{-C_1\eps|\xi|},\\
&\xi\ge0 ~\Rightarrow~C^{-1}\eps e^{-C_2\eps\xi}\le S-u_+ \le C\eps e^{-C_1\eps\xi}.
\end{aligned}
\end{align*} 
Finally, applying the above estimate together with $|S-v_\pm|\le \eps$ to \eqref{gein},
we have \eqref{tail}.\\
As a consequence, \eqref{low-S} follows from the upper bound in \eqref{tail}.\\

$\bullet$ {\it proof of \eqref{sec-S}} :
We first observe that it follows from \eqref{shock} that
\[
S'' =-\s S' +f'(S)S' = \Big( -\s +f'(S)\Big) S'.
\]
Using Taylor theorem with $\eps_0$ small enough, we have
\[
|\s - f'(u_-)| \le C\eps \quad \mbox{by }  \eqref{RH-con},
\]
and 
\[
|f'(S)-f'(u_-)|\le C|S-u_-|\le C\eps.
\]
Therefore, we have
\beq\label{app-sf}
|\s -f'(S)|\le C\eps,
\eeq
which gives the desired estimate.
\end{proof}

\subsection{Relative entropy method}
Our analysis is  based on the relative entropy. The method  is purely nonlinear, and allows to handle rough and large perturbations. The relative entropy method was first introduced by Dafermos \cite{Dafermos1} and Diperna \cite{DiPerna} to prove the $L^2$ stability and uniqueness of Lipschitz solutions to the hyperbolic conservation laws endowed with a convex entropy.

To use the relative entropy method, we first rewrite the viscous term of \eqref{main} into the following form:
\beq\label{eq-0}
u_t + f(u)_x = \Big(\mu(u) \left( \eta'(u)\right)_x \Big)_x
\eeq
where 
\[
\mu(u):=\frac{1}{\eta''(u)}.
\]
Note that $0<\mu(u)\le \alpha^{-1}$ for any $u$ by the hypothesis ($\mathcal{H}1$). 

For simplification of our analysis, we use the change of variable $(t,x)\mapsto (t, \xi=x-\s t)$ to rewrite \eqref{eq-0} into  
\beq\label{eq-1}
u_t -\s u_\xi + f(u)_\xi = \Big(\mu(u) \left( \eta'(u)\right)_\xi \Big)_\xi.
\eeq
Let $A(u):=-\s u +f(u)$. Then \eqref{eq-1} can be rewritten into a simpler form
\beq\label{eq-2}
u_t + A(u)_\xi = \Big(\mu(u) \left( \eta'(u)\right)_\xi \Big)_\xi.
\eeq
Also, \eqref{shock} can be written into
\beq\label{shock-2}
A(S)_\xi = \Big(\mu(S) \left( \eta'(S)\right)_\xi \Big)_\xi.
\eeq

In what follows, we will use the general notation $\mathcal{F}(\cdot|\cdot)$ to denote the relative functional of $\mathcal{F}$ (such as the relative entropy), i.e.,
\[
\mathcal{F}(u|v):= \mathcal{F}(u)-\mathcal{F}(v)-\mathcal{F}'(v)(u-v).
\]
For example, for the fluxes $f$ and $A$, $f(\cdot|\cdot)$ and $A(\cdot|\cdot)$ denote the relative fluxes.\\
By the definition, it is obvious to find
\beq\label{Af}
A(u|v)=f(u|v).
\eeq
On the other hand, we recall the entropy flux $q$ defined by \eqref{def-q}, i.e.,
\[
q'=\eta'f'.
\]
We also let $G$ to denote the entropy flux associated with $A$, i.e.,
\beq\label{def-G}
G'=\eta' A'.
\eeq
Since $A' = -\s + f'$, we find
\[
G' =\eta' A' = \eta' (f' -\s) =q'  -\s \eta' ,
\]
We use the following notations (called fluxes of relative entropy)
\begin{align*}
\begin{aligned}
&q(u;v) = q(u)-q(v) - \eta'(v) (f(u)-f(v)),\\
&G(u;v) = G(u)-G(v) - \eta'(v) (A(u)-A(v)).
\end{aligned}
\end{align*}
Then, we find
\beq\label{Gq}
G(u;v)=q(u;v) -\s \eta(u|v).
\eeq

We will consider a weighted relative entropy between the solution $u$ and the viscous shock $S$ up to a shift $X(t)$ :
\[
a(\xi)\eta\big(u(t,\xi+X(t))| S(\xi)\big),
\]
where the weight $a$ and the shift $X$ will be defined later.\\
The following Lemma provides a quadratic structure on $\frac{d}{dt}\int_{\bbr} a(\xi)\eta\big(U(t,\xi+X(t))|\tilde U_\eps(\xi)\big) d\xi$. 
For simplification, we introduce the following notation: for any function $f : \bbr^+\times \bbr\to \bbr$ and the shift $X(t)$, 
\[
f^{\pm X}(t, \xi):=f(t,\xi\pm X(t)).
\]
We also introduce the functional space
\beq\label{solsp}
\mathcal{H}:=\{u~|~  u \in L^{\infty}(\bbr),~ \partial_\xi \big(\eta'(u)-\eta'(S) \big)\in L^2( \bbr) \},
\eeq
on which the below functionals $Y, \mathcal{B}, \mathcal{G}$ in \eqref{badgood} are well-defined.\\
The space $\mathcal{H}$ will be rigorously handled in Section \ref{global-X}.

\begin{lemma}\label{lem-rel}
Let $a:\bbr\to\bbr^+$ be a smooth bounded function such that $a'$ is bounded. Let $X:\bbr\to\bbr$ be an absolutely continuous function. Let $S$ be the viscous shock in \eqref{shock-2} (or \eqref{shock}). For any solution $u\in\mathcal{H}$ to \eqref{eq-2} (or \eqref{main}), we have 
\begin{align}
\begin{aligned}\label{ineq-1}
\frac{d}{dt}\int_{\bbr} a(\xi)\eta(u^X(t,\xi)|S(\xi)) d\xi =\dot X(t) Y(u^X) +\mathcal{B}(u^X)- \mathcal{G}(u^X),
\end{aligned}
\end{align}
where
\begin{align}
\begin{aligned}\label{badgood}
&Y(u):= -\int_\bbr a'\eta(u|S) d\xi +\int_\bbr a\partial_\xi\eta'(S) (u-S) d\xi,\\
&F(u):= -\int^u \eta''(v)f(v) dv\quad (\mbox{i.e., $F$ is an antiderivative of } -\eta'' f), \\ 
&\mathcal{B}(u):=  \int_\bbr a' F(u|S) d\xi+  \int_\bbr a' \left( \eta'(u) - \eta'(S) \right)\left( f(u) - f(S) \right) d\xi+ \int_\bbr a' f(S) (\eta')(u|S) d\xi\\
&\qquad\qquad - \int_\bbr a\eta''(S)S' f(u|S) d\xi - \int_\bbr a' \mu(u) \left( \eta'(u) - \eta'(S) \right)\partial_\xi  \left( \eta'(u) - \eta'(S) \right) d\xi\\
&\qquad\qquad - \int_\bbr a' \left( \eta'(u) - \eta'(S) \right) \left( \mu(u) - \mu(S) \right) \eta''(S)S' d\xi\\
&\qquad\qquad - \int_\bbr a \partial_\xi  \left( \eta'(u) - \eta'(S) \right) \left( \mu(u) - \mu(S) \right) \eta''(S)S' d\xi + \int_\bbr a S''  (\eta')(u|S) d\xi,\\
&\mathcal{G}(u):=\s \int_\bbr a' \eta(u|S) d\xi +  \int_\bbr a \mu(u) \left|\partial_\xi  \left( \eta'(u) - \eta'(S) \right)\right|^2 d\xi.
\end{aligned}
\end{align}
Note that $(\eta')(u|S)$ denotes $(\eta')(u|S):=\eta'(u)-\eta'(S)-\eta''(S)(u-S)$.
\end{lemma}

\begin{proof}
To derive the desired structure, we  use here a change of variable $\xi\mapsto \xi-X(t)$ as
\[
\int_{\bbr} a(\xi)\eta(u^X(t,\xi)|S(\xi)) d\xi=\int_{\bbr} a^{-X}(\xi)\eta(u(t,\xi)|S^{-X}(\xi)) d\xi.
\]
Let $G$ be the functional defined by \eqref{def-G}.\\
Then, by a straightforward computation together with \cite[Lemma 4]{Vasseur_Book} and the identity $G(u;v)=G(u|v)-\eta'(v)A(u|v)$, we have
\begin{align*}
\begin{aligned}
&\frac{d}{dt}\int_{\bbr} a^{-X}(\xi)\eta(U(t,\xi)|S^{-X}(\xi)) d\xi\\
&=-\dot X \int_{\bbr} \!a'^{-X} \eta(u|S^{-X} ) d\xi +\int_\bbr \!\!a^{-X}\bigg[\Big(\eta'(u)-\eta'(S^{-X})\Big)\!\Big(\!\!\!-\partial_\xi A(u)+ \partial_\xi\Big(\mu(u)\partial_\xi\eta'(u) \Big) \Big)\\
&\qquad -\eta''(S^{-X}) (u-S^{-X}) \Big(-\dot X \partial_\xi S^{-X} -\partial_\xi A(S^{-X})+ \partial_\xi\Big(\mu(S)\partial_\xi\eta'(S^{-X}) \Big)\Big)  \bigg] d\xi\\
&\quad =\dot X \Big( -\int_\bbr a'^{-X}\eta(u|S^{-X}) d\xi +\int_\bbr a^{-X}\partial_\xi \eta'(S^{-X}) (u-S^{-X})  d\xi \Big) +I_1+I_2+I_3,
\end{aligned}
\end{align*}
where 
\begin{align*}
\begin{aligned}
&I_1:=-\int_\bbr a^{-X} \partial_\xi G(u;S^{-X}) d\xi,\\
&I_2:=- \int_\bbr a^{-X} \partial_\xi \eta'(S^{-X}) A(u|S^{-X}) d\xi,\\
&I_3:=\int_\bbr a^{-X} \Big[ \lpa \eta'(u)-\eta'(S^{-X})\rpa   \partial_\xi\lpa\mu(u)\partial_\xi\eta'(u) \rpa \\
&\qquad\qquad\quad -\eta''(S^{-X}) (u-S^{-X}) \partial_\xi\lpa\mu(S^{-X})\partial_\xi\eta'(S^{-X}) \rpa  \Big] d\xi.
\end{aligned}
\end{align*}
Using \eqref{Af} and \eqref{Gq}, we have
\begin{align*}
\begin{aligned}
I_1&=\int_\bbr a'^{-X} G(u;S^{-X}) d\xi = \int_\bbr a'^{-X} q(u;S^{-X}) d\xi  -\sigma \int_\bbr a'^{-X} \eta(u|S^{-X}) d\xi,\\
I_2&=-\int_\bbr a^{-X} \eta''(S^{-X}) S'^{-X} f(u|S^{-X}) d\xi.
\end{aligned}
\end{align*}
For the parabolic part $I_3$, we first rewrite it into
\begin{align*}
\begin{aligned}
I_3&=\int_\bbr a^{-X}  \lpa \eta'(u)-\eta'(S^{-X})\rpa   \partial_\xi\lpa\mu(u)\partial_\xi \lpa \eta'(u) -\eta'(S^{-X}) \rpa  \rpa d\xi\\
&\quad+\int_\bbr a^{-X} \lpa \eta'(u)-\eta'(S^{-X})\rpa   \partial_\xi\lpa \lpa \mu(u) -\mu(S^{-X})\rpa \partial_\xi \eta'(S^{-X}) \rpa d\xi\\
&\quad+\int_\bbr a^{-X}  (\eta')(u|S^{-X})  \partial_\xi\lpa \mu(S^{-X})\partial_\xi \eta'(S^{-X}) \rpa d\xi\\
&=:I_{31}+I_{32}+I_{33}.
\end{aligned}
\end{align*}
Since 
\begin{align*}
\begin{aligned}
I_{31} &=- \int_\bbr a^{-X} \mu(u) \left|\partial_\xi  \left( \eta'(u) - \eta'(S^{-X}) \right)\right|^2 d\xi \\
&\qquad - \int_\bbr a'^{-X} \mu(u) \left( \eta'(u) - \eta'(S^{-X}) \right)\partial_\xi  \left( \eta'(u) - \eta'(S^{-X}) \right) d\xi,\\
I_{32} &= - \int_\bbr a'^{-X} \left( \eta'(u) - \eta'(S^{-X}) \right) \left( \mu(u) - \mu(S^{-X}) \right) \eta''(S^{-X})S'^{-X} d\xi\\
&\qquad - \int_\bbr a^{-X} \partial_\xi  \left( \eta'(u) - \eta'(S^{-X}) \right) \left( \mu(u) - \mu(S^{-X}) \right) \eta''(S^{-X})S'^{-X} d\xi ,\\
I_{33} &= \int_\bbr a^{-X} S''^{-X}  (\eta')(u|S^{-X}) d\xi,\\
\end{aligned}
\end{align*}
we have 
\begin{align*}
\begin{aligned}
&\frac{d}{dt}\int_{\bbr} a^{-X}\eta(u|S^{-X}) d\xi\\
&\quad =\dot X \Big( -\int_\bbr a'^{-X}\eta(u|S^{-X}) d\xi +\int_\bbr a^{-X}\partial_\xi \eta'(S^{-X}) (u-S^{-X}) d\xi \Big)\\
&\qquad + \int_\bbr a'^{-X} q(u;S^{-X}) d\xi  -\sigma \int_\bbr a'^{-X} \eta(u|S^{-X}) d\xi -\int_\bbr a^{-X} \eta''(S^{-X}) S'^{-X} f(u|S^{-X}) d\xi\\
&\qquad - \int_\bbr a^{-X} \mu(u) \left|\partial_\xi  \left( \eta'(u) - \eta'(S^{-X}) \right)\right|^2 d\xi +\int_\bbr a^{-X} S''^{-X}  (\eta')(u|S^{-X}) d\xi \\
&\qquad- \int_\bbr a'^{-X} \mu(u) \left( \eta'(u) - \eta'(S^{-X}) \right)\partial_\xi  \left( \eta'(u) - \eta'(S^{-X}) \right) d\xi\\
&\qquad - \int_\bbr a'^{-X} \left( \eta'(u) - \eta'(S^{-X}) \right) \left( \mu(u) - \mu(S^{-X}) \right) \eta''(S^{-X})S'^{-X} d\xi\\
&\qquad - \int_\bbr a^{-X} \partial_\xi  \left( \eta'(u) - \eta'(S^{-X}) \right) \left( \mu(u) - \mu(S^{-X}) \right) \eta''(S^{-X})S'^{-X} d\xi.
\end{aligned}
\end{align*}
Again, we use a change of variable $\xi\mapsto \xi+X(t)$ to have
\begin{align*}
\begin{aligned}
&\frac{d}{dt}\int_{\bbr} a\eta(u^X|S) d\xi\\
&\quad =\dot X \Big( -\int_\bbr a'\eta(u^X|S) d\xi +\int_\bbr a\partial_\xi \eta'(S) (u^X-S) d\xi  \Big)\\
&\qquad + \int_\bbr a' q(u^X;S) d\xi  -\int_\bbr a \eta''(S) S' f(u^X|S) d\xi +\int_\bbr a S''  (\eta')(u^X|S) d\xi \\
&\qquad- \int_\bbr a' \mu(u^X) \left( \eta'(u^X) - \eta'(S) \right)\partial_\xi  \left( \eta'(u^X) - \eta'(S) \right) d\xi\\
&\qquad - \int_\bbr a' \left( \eta'(u^X) - \eta'(S) \right) \left( \mu(u^X) - \mu(S) \right) \eta''(S)S' d\xi\\
&\qquad - \int_\bbr a \partial_\xi  \left( \eta'(u^X) - \eta'(S) \right) \left( \mu(u^X) - \mu(S) \right) \eta''(S)S' d\xi\\
&\qquad -\sigma \int_\bbr a' \eta(u^X|S) d\xi- \int_\bbr a \mu(u^X) \left|\partial_\xi  \left( \eta'(u^X) - \eta'(S) \right)\right|^2 d\xi.
\end{aligned}
\end{align*}

It remains to rewrite $q(u;S)$ explicitly in terms of $\eta$ and $f$ in quadratic forms. \\
Since
\[
q(u)=\int^u \eta'(v) f'(v) dv,
\]  
using $\eta' f' = (\eta' f)' - \eta'' f$, we have
\begin{align*}
\begin{aligned}
q(u;S) &= q(u)- q(S) - \eta'(S) (f(u)-f(S))\\
&=\int^u_S \eta'(v) f'(v) dv - \eta'(S) (f(u)-f(S))\\
&=-\int^u_S \eta''(v) f(v) dv + \eta'(u)f(u)-\eta'(S)f(S)  - \eta'(S) (f(u)-f(S)).
\end{aligned}
\end{align*}
Then, putting $F(u):= -\int^u \eta''(v)f(v) dv$, we have
\begin{align*}
\begin{aligned}
q(u;S) &=F(u|S) -\eta''(S)f(S) (u-S)  + \Big(\eta'(u)-\eta'(S)\Big)f(S)\\
&=F(u|S) + \Big(\eta'(u)-\eta'(S)\Big)\Big(f(u)-f(S)\Big) + (\eta') (u|S)f(S) .
\end{aligned}
\end{align*}
Hence we have the desired representation \eqref{ineq-1}-\eqref{badgood}.
\end{proof}

\begin{remark}\label{rem:0}
In what follows, we will define the weight function $a$ such that $\sigma a' >0$. Therefore, $\mathcal{G}$ consists of two good terms.
\end{remark}

\subsection{Construction of the weight function}
We first recall that $\s>0$ as mentioned in Remark \ref{rem-s}.\\
We define the weight function $a$ by
\beq\label{weight-a}
a(\xi)=1+\lambda \frac{u_- -S(\xi)}{\eps}.
\eeq
Then, $a' =-\frac{\lambda}{\eps} S' >0$, and thus $\s a' >0$.

\subsection{Global and local estimates on the relative entropy}
We here present useful inequalities on the relative entropy $\eta(\cdot|\cdot)$ that are crucial for the proof of Theorem \ref{thm_general}.  In particular, the specific  coefficients of the estimates \eqref{e-est1}-\eqref{e-est2} will be crucially used in our local analysis on a suitably small truncation.  
\begin{lemma}\label{lem:local}
Let $\eta$ be the entropy satisfying the hypothesis ($\mathcal{H}$1). Then the following 1)-3) hold:\\
1) For any $u, v \in \bbr$,
\beq\label{e-est0}
\eta(u|v) \ge \frac{\alpha}{2} |u-v|^2.
\eeq
2) For any $u_1,u_2,v \in\bbr$ satisfying $v\le u_2 \le u_1$ or $u_1\le u_2 \le v$,
\beq\label{e-mono}
\eta(u_1|v) \ge \eta(u_2|v). 
\eeq
3) For a given constant $u_->0$, there exist positive constants $C$ and $\delta_*$ such that for any $0<\delta<\delta_*$, the following estimates hold:\\
For any $u, v\in \bbr$ with $|u-v|<\delta$ and $|v-u_-|<\delta$,
\begin{align}
\begin{aligned}\label{e-est1}
\eta(u|v)\le \bigg(\frac{\eta''(v)}{2}  + C\delta \bigg) |u-v|^2,
\end{aligned}
\end{align}
\beq\label{e-est2}
\eta(u|v)\ge \frac{\eta''(v)}{2} |u-v|^2 + \frac{\eta'''(v)}{6} (u-v)^3,
\eeq
For any $u, v\in \bbr$ satisfying $|v-u_-|< \delta$ and either $\eta(u|v)\le \delta$ or $|\eta'(u)-\eta'(v)|\le \delta$,
\beq\label{e-est3}
|\eta'(u)-\eta'(v)|^2 \le C\eta(u|v).
\eeq
\end{lemma}
\begin{proof}
$\bullet$ {\it proof of \eqref{e-est0}} :
Since 
\beq\label{releq}
\eta(u|v)=(u-v)^2\int_0^1\int_0^1 \eta''(v + st(u-v)) t\, dsdt,
\eeq
using $\eta''\ge\alpha$ by ($\mathcal{H}$1), we find
\[
\eta(u|v)\ge (u-v)^2 \alpha \int_0^1\int_0^1 t\, dsdt,
\]
which gives \eqref{e-est0}.\\

$\bullet$ {\it proof of \eqref{e-mono}} : 
Since $u\mapsto \eta(u|v)$ is convex in $u>0$ and zero at $u=v$, we see that $u\mapsto \eta(u|v)$ is increasing in $|u-v|$, which implies \eqref{e-mono}.\\

$\bullet$ {\it proof of \eqref{e-est1}-\eqref{e-est2}} :
Since 
\[
\eta(u|v)=\eta(u)-\eta(v)-\eta'(v)(u-v),
\]
applying Taylor theorem to $\eta$ about $v$, we have
\beq\label{et-instant}
\eta(u|v)=\frac{\eta''(v)}{2} (u-v)^2 +  \frac{\eta'''(v)}{6} (u-v)^3 +  \frac{\eta''''(v_*)}{24} (u-v)^4,
\eeq
where $v_*$  lies between $u$ and $v$.\\
First, the hypothesis $\eta''''\ge \alpha$ of ($\mathcal{H}$1) implies \eqref{e-est2}.\\
Secondly, since $|u-v|<\delta$ and $|v-u_-|<\delta$, there exists a positive constant $C$ (depending only on $u_-$) such that $|\eta'''(v)|+|\eta''''(v)|\le C$. Therefore \eqref{e-est1} follows from \eqref{e-est1}.\\

$\bullet$ {\it proof of \eqref{e-est3}} :
If $\eta(u|v)\le\delta$, then \eqref{e-est0} yields $|u-v|^2 \le  \frac{2}{\alpha} \delta$. If $|\eta'(u)-\eta'(v)|\le \delta$, then the assumption $\eta''\ge\alpha$ yields
$|u-v|\le \frac{1}{\alpha} \delta$.\\
Therefore, using the mean-value theorem together with $|v-u_-|\le \delta < \delta_*<1$, we find
\[
|\eta'(u)-\eta'(v)|^2 \le C|u-v|^2,
\]
which together with \eqref{e-est0} implies \eqref{e-est3}.
\end{proof}

\begin{remark}
Notice that we only need the hypothesis ($\mathcal{H}$1) on entropy for Lemma \ref{lem:local}. In fact, the remaining hypothesis ($\mathcal{H}$2) is crucially used in the proof of Proposition \ref{prop_out}.
\end{remark}


\vspace{1cm}

\section{Proofs of Theorem \ref{thm_general} and Corollary \ref{coro}}\label{section_theo} 
\setcounter{equation}{0}

\subsection{Definition of the shift}

For any fixed $\eps>0$, we consider a continuous function $\Phi_\eps$ defined by
\beq\label{Phi-d}
\Phi_\eps (y)=
\left\{ \begin{array}{ll}
      \frac{1}{\eps^2},\quad \mbox{if}~ y\le -\eps^2, \\
      -\frac{1}{\eps^4}y,\quad \mbox{if} ~ |y|\le \eps^2, \\
       -\frac{1}{\eps^2},\quad \mbox{if}  ~y\ge \eps^2. \end{array} \right.
\eeq
We define a shift function $X(t)$ as a solution of the nonlinear ODE:
\beq\label{X-def}
\left\{ \begin{array}{ll}
       \dot X(t) = \Phi_\eps (Y(u^X)) \Big(2|\mathcal{B}(u^X)|+1 \Big),\\
       X(0)=0, \end{array} \right.
\eeq
where $Y$ and $\mathcal{B}$ are as in \eqref{badgood}.

\subsection{Existence of the shift}\label{global-X}
First of all, we note that our initial condition $\int_\bbr \eta(u_0|S) dx <\infty$ together with the property \eqref{e-est0} implies $u_0-S\in L^2(\bbr)$. It is known that if the initial data $u_0\in L^\infty(\bbr)$ and $u_0-S\in L^2(\bbr)$ then for any $T>0$, \eqref{main} admits a bounded weak solution $u$ such that $\|u\|_{L^\infty(\bbr)} \le \|u_0\|_{L^\infty(\bbr)}$ and
\beq\label{2est}
\int_\bbr |u-S|^2 d\xi +\int_0^t\int _\bbr |(u-S)_\xi |^2 d\xi \le e^{Ct}   \int_\bbr |u_0-S|^2 d\xi,\quad t\le T.
\eeq
Indeed, the first inequality follows from the maximum principle of the scalar conservation law. The second estimate \eqref{2est} is obtained by the energy method as follows: Using \eqref{eq-1} and \eqref{eq-2} together with $\|u\|_{L^\infty(\bbr)} \le \|u_0\|_{L^\infty(\bbr)}$, we have
\begin{align*}
\begin{aligned}
&\frac{d}{dt} \frac{1}{2}\int_{\bbr} |u-S|^2 d\xi + \int _\bbr |(u-S)_\xi|^2 d\xi  =   \int _\bbr \left(A(u)-A(S) \right) (u-S)_\xi d\xi \\
&\qquad \le \frac{1}{2}  \int _\bbr |(u-S)_\xi|^2 d\xi + \frac{1}{2}  \int _\bbr |A(u)-A(S) |^2 d\xi\\
&\qquad \le \frac{1}{2}  \int _\bbr |(u-S)_\xi|^2 d\xi + C \int _\bbr |u-S |^2 d\xi,
\end{aligned}
\end{align*}
which together with Gronwall inequality implies \eqref{2est}.\\

Then, using $S'\in L^2(\bbr)$ together with $\|u\|_{L^\infty(\bbr)} \le \|u_0\|_{L^\infty(\bbr)}$, we find 
\[
\left(\eta'(u)-\eta'(S)\right)_\xi =\eta''(u) (u-S)_\xi +\left( \eta''(u) -\eta''(S) \right) S' \in L^2((0,T)\times\bbr).
\]
Thus, $u\in \mathcal{H}$ (recall \eqref{solsp}) for a.e. $t\le T$, all functionals $Y, \mathcal{B}, \mathcal{G}$ in \eqref{badgood} are well-defined. \\
We now guarantee the existence of the shift $X$ as follows.
\begin{lemma}
Let $u$ be any function of $(t,x)$ such that 
\beq\label{odes}
\|u\|_{L^\infty(\bbr)} \le \|u_0\|_{L^\infty(\bbr)},\quad \left(\eta'(u)-\eta'(S)\right)_\xi \in L^2((0,T)\times\bbr).
\eeq
Then the ODE \eqref{X-def} has a unique absolutely continuous solution $X$ on $[0,T]$.
\end{lemma}
\begin{proof}
We will use the following lemma (see \cite[Lemma A.1]{CKKV}) which is a simple adaptation of the well-known Cauchy-Lipschitz theorem.  \\

 \begin{lemma} \cite[Lemma A.1]{CKKV} \label{lem_ckkv}
 Let $p>1$ and $T>0$. Suppose that a function 
 $F:[0,T]\times\bbr\rightarrow\bbr$  satisfies 
 $$\sup_{x\in\bbr }|F(t,x)|\leq a(t) \quad\mbox{and}\quad
\sup_{x,y\in\bbr,x\neq y }\Big|\frac{F(t,x)-F(t,y)}{x-y}\Big|\leq b(t)
\quad \mbox{for } t\in[0,T] $$ for some functions $a \in L^1(0,T)$ and $\, b\in L^p(0,T)$. Then for any $x_0\in\bbr$, there exists a unique absolutely continuous function $X:[0,T]\rightarrow \bbr$ satisfying
\beq\label{ode_eq}\left\{ \begin{array}{ll}
        \dot X(t) = F(t,X(t))\quad\mbox{for \textit{a.e.} }t\in[0,T],\\
       X(0)=x_0\end{array} \right.\eeq
 \end{lemma}

To apply the above lemma, let $F(t,X)$ denote the right-hand side of the ODE \eqref{X-def}. \\
First of all, we move the shift $X$ in $u$ into the other smooth functions $S, a, a', S''$ in the integrands of the functionals $Y$ and $\mathcal{B}$, by using the change of variables $\xi\mapsto \xi-X(t)$ as done in the proof of Lemma \ref{lem-rel}.\\
Then, we first observe that  \eqref{Phi-d} and \eqref{X-def} imply
\[
|F(t,x)|\le \frac{1}{\eps^2} \Big(2|\mathcal{B}(u^x)|+1 \Big),\quad \forall t\in[0,T],\quad \forall x\in\bbr.
\]
It is easy to show that $\sup_{x\in \bbr} |\mathcal{B}(u^x)|\le h(t)$ for all $t\in[0,T]$, for some $h\in L^2([0,T])\subset L^1([0,T])$. Indeed, for example, we estimate the fifth term of $\mathcal{B}(u^x)$ in \eqref{badgood} as follows: using \eqref{odes}, and $a'\in L^\infty\cup L^1$,
\begin{align*}
\begin{aligned}
&\Big| \int_\bbr a'^{-x} \mu(u) \left( \eta'(u) - \eta'(S^{-x}) \right)\partial_\xi  \left( \eta'(u) - \eta'(S^{-x}) \right) d\xi \Big| \le C \int_\bbr a'^{-x} \left|\partial_\xi  \left( \eta'(u) - \eta'(S^{-x}) \right) \right| d\xi \\
&\qquad \le  C \int_\bbr a'^{-x} \Big( \left|\partial_\xi  \left( \eta'(u) - \eta'(S) \right) \right| +\left|\partial_\xi  \left( \eta'(S) - \eta'(S^{-x}) \right) \right| \Big) d\xi \\
&\qquad \le  C \int_\bbr a'^{-x} \Big( \left|\partial_\xi  \left( \eta'(u) - \eta'(S) \right) \right| +C \Big) d\xi  \\
&\qquad \le  C+ C\sqrt{\int_\bbr a'^{-x} d\xi}  \sqrt{\int_\bbr  \left|\partial_\xi  \left( \eta'(u) - \eta'(S) \right) \right|^2 d\xi},
\end{aligned}
\end{align*}
where the all constants $C$ are independent of $x$, and the right-hand side belongs to $ L^1([0,T])$. Likewise, we can control other terms of $\mathcal{B}(u^x)$ by the time-integrable function.\\
Similarly, using \eqref{odes}, we can show that
\[
\sup_{x\in \bbr} |\partial_xF(t,x)|\le m(t),
\]
for some $\|m\|_{L^2([0,T])}\le C_*$.  
Therefore, we completes the proof by Lemma \ref{lem_ckkv}.
\end{proof}

\subsection{Proof of Theorem~\ref{thm_general} from a main proposition} \label{subsec:thm}
First of all, our main proposition is the following.
\begin{proposition}\label{prop:main}
There exist $\delta_0\in (0,1)$ such that for any $\eps, \lambda$ with $\delta_0^{-1}\eps<\lambda<\delta_0<1/2$, the following is true.\\
For any $u\in \{u~|~|Y(u)|\le\eps^2 \}$,
\beq\label{prop:est}
\mathcal{R}(u):= -\frac{1}{\eps^4}Y^2(u) +\mathcal{B}(u)+\delta_0 \frac{\eps}{\lambda} |\mathcal{B}(u)| - \mathcal{G}(u) \le 0.
\eeq
\end{proposition}
 
We will first show how this proposition implies Theorem~\ref{thm_general} as follows:

Based on \eqref{ineq-1} and \eqref{X-def}, to get the contraction estimate \eqref{cont_main}, it is enough to prove that for almost every time $t>0$ 
\beq\label{contem0}
\Phi_\eps (Y(u^X)) \Big(2|\mathcal{B}(u^X)|+1 \Big) Y(u^X) +\mathcal{B}(u^X)-\mathcal{G}(u^X)\le0.
\eeq
For every $u$ we define 
\beq\label{rhs}
\mathcal{F}(U):=\Phi_\eps (Y(u)) \Big(2|\mathcal{B}(u)|+1 \Big)Y(u) +\mathcal{B}(u)-\mathcal{G}(u).
\eeq
From \eqref{Phi-d}, we have 
\beq\label{XY}
\Phi_\eps (Y) \Big(2|\mathcal{B}|+1 \Big)Y\le
\left\{ \begin{array}{ll}
     -2|\mathcal{B}|,\quad \mbox{if}~  |Y|\ge \eps^2,\\
     -\frac{1}{\eps^4}Y^2,\quad  \mbox{if}~ |Y|\le \eps^2. \end{array} \right.
\eeq
Hence,   for all $u$ satisfying  $|Y(u)|\ge \eps^2 $, we have 
$$
\mathcal{F}(u) \le -|\mathcal{B}(u)|-\mathcal{G}(u) \le 0.
$$
Using both (\ref{XY}) and Proposition \ref{prop:main}, we find that for all $u$ satisfying  $|Y(u)|\le \eps^2 $, 
$$
\mathcal{F}(u) \le -\delta_0\left(\frac{\eps}{\lambda}\right)|\mathcal{B}(u)| \le 0.
$$
 Since $\delta_0 <1$ and $\eps/\lambda<\delta_0$,   these two estimates show that for every $u$ we have 
 $$
\mathcal{F}(u) \le -\delta_0\left(\frac{\eps}{\lambda}\right)|\mathcal{B}(u)|.
$$
For every fixed $t>0$, using this estimate with $u=u^X(t,\cdot)$, together with  \eqref{ineq-1}, and \eqref{contem0} gives 
\beq\label{111}
\frac{d}{dt}\int_{\bbr} a\eta(u^X|S) d\xi \le \mathcal{F}(u^X) \le -\delta_0\left(\frac{\eps}{\lambda}\right)|\mathcal{B}(u^X)|.
\eeq
Thus, $\frac{d}{dt}\int_{\bbr} a\eta(u^X|S) d\xi \le0$, which completes \eqref{cont_main}.\\
Moreover, since it follows from \eqref{111} that
\[
\delta_0\left(\frac{\eps}{\lambda}\right)\int_0^{T}|\mathcal{B}(u^X)|dt \le  \int_{\bbr} \eta(u_0|S) d\xi <\infty \quad\mbox{by the initial condition},
\]
using \eqref{X-def} and $\|\Phi_\eps\|_{L^{\infty}(\bbr)}\le \frac{1}{\eps^2}$ by \eqref{Phi-d}, we have
\beq\label{112}
|\dot X|\le \frac{1}{\eps^2} + \frac{2}{\eps^2}|\mathcal{B}|, \quad  \|\mathcal{B}\|_{L^1(0,T)}\le \frac{1}{\delta_0}\frac{\lambda}{\eps}\int_{\bbr} \eta(u_0|S) d\xi.
\eeq
This provides the global-in-time estimate \eqref{est-shift}, and thus $X\in W^{1,1}_{loc}((0,T))$. This completes the proof of Theorem \ref{thm_general}.
\vskip0.1cm Section \ref{sec:prop} is dedicated to the proof of Proposition \ref{prop:main}.

\subsection{Proof of Corollary \ref{coro}}
First of all, following the proof of Proposition \ref{prop:main}, we observe that the hypothesis ($\mathcal{H}$2) is used only for \eqref{use1},  \eqref{use2},  \eqref{use3},  \eqref{use4},  \eqref{use5},  \eqref{use6} and  \eqref{use7} in the proof of Proposition \ref{prop_out}, from which the constant $\theta>0$ should be chosen such that $\theta\ge 3|u_-|$.\\
Therefore, for any $M>0$ such that $M\ge 6|u_-|$, since it follows from the maximum principle that $\|u\|_{L^\infty([0,T]\times\bbr)} \le \|u_0\|_{L^\infty(\bbr)}\le M$, it is enough to use the hypothesis ($\mathcal{H}$2) in the case of $|u|\le M$ and $|v|\le M/2$, which means  the new hypothesis ($\mathcal{H}2'$). Hence we have Corollary \ref{coro}.

\section{Proof of Proposition \ref{prop:main}}\label{sec:prop}
In this section, we present the proof of Proposition \ref{prop:main}, which was a main part of the proof of Theorem \ref{thm_general} in Section \ref{subsec:thm}. \\
\subsection{Ideas of Proof}
As in \cite{Kang-V-NS17}, the main idea is to use the generalized Poincar\'e inequality in Lemma \ref{lem_poincare}. For that, we first truncate the perturbation $|\eta'(u)-\eta'(S)|$ with a small constant $\delta_1>0$, due to the structure of the diffusion term $\mathcal{D}$ (second term of $\mathcal{G}$ in \eqref{badgood}), where the constant $\delta_1$ is determined by  the generalized Poincar\'e inequality. For values of $u$ such that $|\eta'(u)-\eta'(S)|\le \delta_1\ll 1$ (i.e., near $S$), we do the Taylor expansions at $S$ for the all terms of the functionals $Y, F$, $\mathcal{G}$, and the first four terms (denoted by $\mathcal{B}_1,...,\mathcal{B}_4$ as in \eqref{badgood-n}) of $\mathcal{B}$ in \eqref{badgood}. It will be proved in Proposition \ref{prop:near} that the leading order term of such expansions is non-positive thanks to Lemma \ref{lem_poincare}. 
The reason why we consider only the four terms $\mathcal{B}_1,...,\mathcal{B}_4$ in Proposition \ref{prop:near}  is that the remaining bad terms (denoted by $\mathcal{B}_5,...,\mathcal{B}_8$ as in \eqref{newbad})  is negligible by the good terms (see \eqref{n2}).\\
Therefore, it remains to show that for values of $u$ such that $|\eta'(u)-\eta'(S)|>\delta_1$, the all terms are negligible by the good terms. This will be proved in Proposition \ref{prop_out}. \\
Finally, in Section \ref{sub:con}, we combine the main Propositions \ref{prop:near} and \ref{prop_out} to complete the proof of Proposition \ref{prop:main}.

\vskip0.5cm

Now, we begin by recalling the functionals $Y$ and $F$ in \eqref{badgood}, and define the following functionals:
\begin{align}
\begin{aligned}\label{badgood-n}
&Y(u):= -\int_\bbr a'\eta(u|S) d\xi +\int_\bbr a\partial_\xi\eta'(S) (u-S) d\xi,\\
&F(u):= -\int^u \eta''(v)f(v) dv\quad (\mbox{i.e., $F$ is an antiderivative of } -\eta'' f), \\ 
&\mathcal{B}_1(u):=  \int_\bbr a' F(u|S) d\xi,\\
&\mathcal{B}_2(u):=  \int_\bbr a' \left( \eta'(u) - \eta'(S) \right)\left( f(u) - f(S) \right) d\xi,\\
&\mathcal{B}_3(u):=  \int_\bbr a' f(S) (\eta')(u|S) d\xi,\\
&\mathcal{B}_4(u):=- \int_\bbr a\eta''(S)S' f(u|S) d\xi,\\
&\mathcal{G}_0(u):=\s \int_\bbr a' \eta(u|S) d\xi , \\
&\mathcal{D}(u):=  \int_\bbr a \mu(u) \left|\partial_\xi  \left( \eta'(u) - \eta'(S) \right)\right|^2 d\xi.
\end{aligned}
\end{align}

\begin{proposition}\label{prop:near}
Let $\eta$ be the entropy satisfying the hypothesis ($\mathcal{H}$1). 
For any $K>0$, there exists $\delta_{1}\in(0,1)$ such that  for any $\delta_1^{-1}\eps<\lambda<\delta_1$ and for  any $\delta\in(0,\delta_{1})$,   
 the following is true.\\
For any function $u:\bbr\to \bbr$ such that $\mathcal{D}(u)+\mathcal{G}_0(u)$
 is finite, if
\beq\label{assYp}
|Y(u)|\leq K \frac{\eps^2}{\lambda},\qquad  \|\eta'(u)-\eta'(S)\|_{L^\infty(\bbr)}\leq \delta_1,
\eeq
then
\begin{align}
\begin{aligned}\label{redelta}
&\mathcal{R}_{\eps,\delta}(u):=-\frac{1}{\eps\delta}|Y(u)|^2 +\mathcal{B}_1(u)+\mathcal{B}_2(u)+\mathcal{B}_3(u)+\mathcal{B}_4(u) \\
&\qquad\qquad\qquad + \delta\frac{\eps}{\lambda}|\mathcal{B}_1(u)+\mathcal{B}_2(u)+\mathcal{B}_3(u)+\mathcal{B}_4(u)| -\left(1-\delta\frac{\eps}{\lambda}\right)\mathcal{G}_0(u)-(1-\delta)\mathcal{D}(u)\le 0.
\end{aligned}
\end{align}
\end{proposition}

To prove this proposition, we will use the nonlinear Poincar\'e type inequality in \cite{Kang-V-NS17}:
\begin{lemma}\cite[Proposition 3.3]{Kang-V-NS17}\label{lem_poincare}
For a given $M>0$, there exists $\delta^{*}=\delta^{*}(M)>0$ such that for any $\delta\in(0,\delta^{*})$ the following is true:\\ For any $W\in L^2(0,1)$ with  
$\sqrt{y(1-y)}\partial_yW\in L^2(0,1)$, if $\int_0^1 |W(y)|^2\,dy\leq M$, then
\beq\label{inequal_poincare}
 \mathcal{R}_\delta(W)\leq0.
\eeq where  
\beq\label{def_poincare}\begin{split}
\mathcal{R}_\delta(W):&=-\frac{1}{\delta}\left(\int_0^1W^2\,dy+2\int_0^1 W\,dy\right)^2+(1+\delta)\int_0^1 W^2\,dy\\
&\quad+\frac{2}{3}\int_0^1 W^3\,dy 
+\delta \int_0^1 |W|^3\,dy 
-(1-\delta)\int_0^1 y(1-y)|\partial_y W|^2\,dy.
\end{split}\eeq
\end{lemma}

\subsection{Proof of Proposition \ref{prop:near}}
We first observe that since $\eta''\ge \alpha$ by the hypothesis ($\mathcal{H}$1), the mean-value theorem implies
\beq\label{instant0}
\|u-S\|_{L^\infty(\bbr)}\le \alpha^{-1}  \|\eta'(u)-\eta'(S)\|_{L^\infty(\bbr)}.
\eeq
We take $\delta_1\in (0,1)$ small enough such that
\[
\delta_1\le\min(\alpha,1)  \delta_*,
\]
where $\delta_*$ is the constant as in Lemma \ref{lem:local}.\\
Then it follows from \eqref{instant0} that
\beq\label{us1}
\|u-S\|_{L^\infty(\bbr)}\le\alpha^{-1}\delta_1 \le \delta_*.
\eeq
Moreover, since 
\beq\label{us2}
|S-u_-|\le \eps=\frac{\eps}{\lambda}\lambda<\delta_1^2 \le \delta_1\le \delta_*,
\eeq
we can use Lemma \ref{lem:local} in this proof.\\

In what follows, we will rewrite the functionals in \eqref{badgood-n} and $Y$  in terms of $y$ variable defined by
\beq\label{w1w2}
y(\xi)=\frac{u_- -S(\xi)}{\eps}.
\eeq
Notice that since $S'<0$, we can use a change of variable $\xi\in\bbr\mapsto y\in[0,1]$.\\
Then it follows from \eqref{weight-a} that $a=1+\lambda y$ and
\beq\label{ach}
a'(\xi)=- \frac{\lambda}{\eps} S'(\xi) =\lambda \frac{dy}{d\xi}.
\eeq
For simple presentation and normalization, we will use the notations
\beq\label{wW}
w(y):=(u-S)\circ y^{-1},\quad W:=\frac{\lambda}{\eps}w.
\eeq


$\bullet$ {\bf Change of variable for $Y$:} 
We first recall 
\[
Y= \underbrace{-\int_\bbr a'\eta(u|S) d\xi }_{=:Y_1} +\underbrace{\int_\bbr a\, \eta''(S) S' (u-S) d\xi }_{=:Y_2}.
\]
Using \eqref{e-est1} and \eqref{e-est2} in Lemma  \ref{lem:local}, we have
$$
\left| Y_1 + \int_\bbr a' \frac{\eta''(S)}{2}(u-S)^2 d\xi  \right|\leq C\delta_1 \int_\bbr a' (u-S)^2 d\xi.
$$
Using
\beq\label{small2}
|\eta''(S)-\eta''(u_-)|\le C\eps \le C\delta_1,
\eeq
and \eqref{ach} with notation \eqref{wW}, we get
\begin{align}\label{alpha1}
\begin{aligned}
\left|Y_1 + \lambda \frac{\eta''(u_-)}{2}\int_0^1 w^2 dy\right|\leq C\lambda\delta_1 \int_0^1 w^2 dy.
\end{aligned}
\end{align}
Likewise, using  \eqref{ach} and \eqref{small2} together with $|a-1|\le\lambda<\delta_1$, we have
\[
\left|Y_2 +\eps\eta''(u_-)  \int_0^1  w dy\right|\leq C\eps\delta_1 \int_0^1 |w| dy. 
\]
Therefore, we have
\[
\left|Y + \frac{\eta''(u_-)}{2} \Big( \lambda \int_0^1 w^2 dy +2\eps \int_0^1  w dy \Big)\right| \le C\delta_1 \Big( \lambda \int_0^1 w^2 dy +\eps \int_0^1  |w| dy \Big).
\]
Using the notation $W$ in \eqref{wW}, we have
\[
\left|Y + \frac{\eta''(u_-)}{2}  \frac{\eps^2}{\lambda} \Big(  \int_0^1 W^2 dy +2 \int_0^1  W dy \Big)\right| \le C\delta_1 \frac{\eps^2}{\lambda} \Big( \int_0^1 W^2 dy + \int_0^1  |W| dy \Big),
\]
which gives
\beq\label{YW}
\left| \frac{2}{\eta''(u_-)} \frac{\lambda}{\eps^2} Y +   \int_0^1 W^2 dy +2 \int_0^1  W dy \right| \le C\delta_1 \Big( \int_0^1 W^2 dy + \int_0^1  |W| dy \Big).
\eeq

\vskip0.3cm

$\bullet$ {\bf Change of variable for  $\mathcal{B}_1, \mathcal{B}_2, \mathcal{B}_3$ and $\mathcal{B}_4$:} 
First of all, we have
\begin{align*}
\begin{aligned}
&\mathcal{B}_1+\mathcal{B}_2+\mathcal{B}_3\\
&\quad=  \int_\bbr a' \Big(F(u|S)+  \left( \eta'(u) - \eta'(S) \right)\left( f(u) - f(S) \right) + f(S) (\eta')(u|S)  \Big) d\xi.
\end{aligned}
\end{align*}
Since
\[
F(u|S) =F(u)-F(S)-F'(S)(u-S),
\]
using Taylor theorem together with
\[
F''=-\eta''' f -\eta'' f', \quad F'''=-\eta'''' f -2 \eta''' f' - \eta'' f'',
\]
we find that for any $u$ satisfying \eqref{us1} and \eqref{us2}, 
\begin{align*}
\begin{aligned}
&\Big|F(u|S) +\frac{1}{2} \Big(\eta'''(S) f(S) +\eta''(S) f'(S) \Big) (u-S)^2 \\
&\quad +\frac{1}{6} \Big(\eta''''(S) f(S) +2\eta'''(S) f(S)+\eta''(S) f''(S) \Big) (u-S)^3 \Big| \le C\delta_1 |u-S|^3.
\end{aligned}
\end{align*}
Likewise, we find that for any $u$ satisfying \eqref{us1} and \eqref{us2}, 
\begin{align*}
\begin{aligned}
&\Big| \left( \eta'(u) - \eta'(S) \right)\left( f(u) - f(S) \right)  -\eta''(S)f'(S)(u-S)^2 \\
&\quad -\frac{1}{2} \Big(\eta''(S) f''(S) +\eta'''(S) f'(S)  \Big) (u-S)^3\Big| \le  C\delta_1 |u-S|^3, 
\end{aligned}
\end{align*}
and
\begin{align*}
\begin{aligned}
\Big|  f(S) (\eta')(u|S)  -f(S)\frac{\eta''(S)}{2} (u-S)^2 - f(S) \frac{\eta'''(S) }{6}(u-S)^3\Big| \le  C\delta_1 |u-S|^3.
\end{aligned}
\end{align*}
Therefore, we have
\begin{align}
\begin{aligned}\label{mb123}
&\bigg|  \mathcal{B}_1+\mathcal{B}_2+\mathcal{B}_3 \\
&\qquad -\int_\bbr a'\Big[ \frac{1}{2}\eta''(S) f'(S)  (u-S)^2 +\Big(  \frac{1}{3}\eta''(S) f''(S) + \frac{1}{6}\eta'''(S) f'(S) \Big)(u-S)^3 \Big] d\xi \bigg|   \\
&\quad \le  C\delta_1 \int_\bbr a'  |u-S|^3  d\xi.
\end{aligned}
\end{align}
Using \eqref{ach} together with \eqref{small2} and
\beq\label{minor0}
|f'(S)-f'(u_-)|\le C\eps,\quad |f''(S)-f''(u_-)|\le C\eps,\quad |\eta'''(S)-\eta'''(u_-)|\le C\eps,
\eeq
we have
\begin{align}
\begin{aligned}\label{B123}
 \mathcal{B}_1+\mathcal{B}_2+\mathcal{B}_3  &\le   \lambda \frac{1}{2}\eta''(u_-) f'(u_-)\int_0^1 w^2 dy +C \lambda\eps\int_0^1 w^2 dy \\
&\quad+ \lambda\Big(\frac{1}{3}\eta''(u_-) f''(u_-) + \frac{1}{6}\eta'''(u_-) f'(u_-) \Big) \int_0^1 w^3 dy + C\lambda\delta_1\int_0^1 |w|^3 dy.
\end{aligned}
\end{align}

Likewise, since
\beq\label{mb4}
\Big|\mathcal{B}_4 + \int_\bbr a S'  \eta''(S)\frac{f''(S)}{2} (u-S)^2  d\xi \Big| \le  C\delta_1 \int_\bbr |S'|  |u-S|^2  d\xi,
\eeq
using $|a-1|\le\lambda<\delta_1$, we have
\[
\mathcal{B}_4 \le  \eps \frac{1}{2}\eta''(u_-)f''(u_-) \int_0^1 w^2 dy + C\eps\delta_1  \int_0^1 w^2 dy.
\]
Therefore, this and \eqref{B123} yields
\begin{align}
\begin{aligned}\label{B1234}
& \mathcal{B}_1+\mathcal{B}_2+\mathcal{B}_3 +\mathcal{B}_4 \\
 &\quad\le   \lambda \frac{1}{2}\eta''(u_-) f'(u_-)\int_0^1 w^2 dy + \eps \frac{1}{2}\eta''(u_-)f''(u_-) \int_0^1 w^2 dy + C\eps\delta_1  \int_0^1 w^2 dy \\
&\quad\quad+ \lambda\Big(\frac{1}{3}\eta''(u_-) f''(u_-) + \frac{1}{6}\eta'''(u_-) f'(u_-) \Big) \int_0^1 w^3 dy + C\lambda\delta_1\int_0^1 |w|^3 dy.
\end{aligned}
\end{align}

\vskip0.3cm

$\bullet$ {\bf Change of variable for $\mathcal{G}_0$:}
We first use \eqref{e-est2} in Lemma \ref{lem:local} to get
\[
\mathcal{G}_0=\s \int_\bbr a' \eta(u|S) d\xi \ge \s \int_\bbr a' \Big(\frac{\eta''(S)}{2} |u-S|^2 + \frac{\eta'''(S)}{6} (u-S)^3 \Big) d\xi.
\]
Then using \eqref{app-sf}, we have
\begin{align*}
\begin{aligned}
\mathcal{G}_0&\ge \left( \frac{f'(u_-)\eta''(u_-)}{2} -C\eps\right) \int_\bbr a'  |u-S|^2 d\xi \\
&\quad+ \frac{f'(u_-)\eta'''(u_-)}{6}   \int_\bbr a'  (u-S)^3 d\xi -C\eps \int_\bbr a'  |u-S|^3 d\xi.
\end{aligned}
\end{align*}
Thus,
\beq\label{G0}
-\mathcal{G}_0 \le -\lambda\left( \frac{f'(u_-)\eta''(u_-)}{2} -C\eps\right)\int_0^1 w^2 dy - \lambda\frac{f'(u_-)\eta'''(u_-)}{6}  \int_0^1 w^3 dy +C\lambda\eps \int_0^1 |w|^3 dy.
\eeq

\vskip0.3cm
$\bullet$ {\bf Estimates on $ \mathcal{B}_1+\mathcal{B}_2+\mathcal{B}_3 +\mathcal{B}_4 -\mathcal{G}_0$:} 
We combine \eqref{B1234} and \eqref{G0} to have
\begin{align*}
\begin{aligned}
 \mathcal{B}_1+\mathcal{B}_2+\mathcal{B}_3 +\mathcal{B}_4 -\mathcal{G}_0 &\le   \eps \frac{1}{2}\eta''(u_-)f''(u_-) \int_0^1 w^2 dy + C\eps\delta_1  \int_0^1 w^2 dy \\
&\quad+ \lambda\frac{1}{3}\eta''(u_-) f''(u_-)  \int_0^1 w^3 dy + C\lambda\delta_1\int_0^1 |w|^3 dy.
\end{aligned}
\end{align*}
Using the notation $W$ in \eqref{wW}, we have
\begin{align*}
\begin{aligned}
 \mathcal{B}_1+\mathcal{B}_2+\mathcal{B}_3 +\mathcal{B}_4 -\mathcal{G}_0 &\le   \frac{\eps^3}{\lambda^2}\frac{\eta''(u_-)f''(u_-)}{2} \Big( 
 \int_0^1 W^2 dy +\frac{2}{3} \int_0^1 W^3 dy \Big) \\
&\quad + C \frac{\eps^3}{\lambda^2} \delta_1  \Big( \int_0^1 W^2 dy +    \int_0^1 |W|^3\, dy\Big).
\end{aligned}
\end{align*}
Since $\eta''(u_-)f''(u_-)>0$, the above estimate can be rewritten into (in a way of normalizing the right-hand side) :
\begin{align}
\begin{aligned}\label{sum-bg}
& \frac{2}{\eta''(u_-)f''(u_-)}\frac{\lambda^2}{\eps^3} \Big( \mathcal{B}_1+\mathcal{B}_2+\mathcal{B}_3 +\mathcal{B}_4 -\mathcal{G}_0 \Big) \\
&\quad\le  \int_0^1 W^2 dy + \frac{2}{3}   \int_0^1 W^3\, dy + C  \delta_1  \Big( \int_0^1 W^2 dy + \int_0^1 |W|^3\, dy\Big).
\end{aligned}
\end{align}

On the other hand, using \eqref{mb123}, \eqref{mb4} and \eqref{e-est1}, we have a rough estimate:
\[
|\mathcal{B}_1+\mathcal{B}_2+\mathcal{B}_3+\mathcal{B}_4| +\mathcal{G}_0 \le C\lambda \int_0^1 w^2 dy = C\frac{\eps^2}{\lambda} \int_0^1 W^2 dy,
\]
which yields
\[
\delta_1\frac{\eps}{\lambda}  \Big(|\mathcal{B}_1+\mathcal{B}_2+\mathcal{B}_3+\mathcal{B}_4| +\mathcal{G}_0\Big) \le C\delta_1\frac{\eps^3}{\lambda^2} \int_0^1 W^2 dy.
\]
Therefore, we have
\beq\label{sum-a}
  \frac{2}{\eta''(u_-)f''(u_-)}\frac{\lambda^2}{\eps^3}\Big[ \delta_1\frac{\eps}{\lambda}  \Big(|\mathcal{B}_1+\mathcal{B}_2+\mathcal{B}_3+\mathcal{B}_4| +\mathcal{G}_0 \Big) \Big] \le C\delta_1 \int_0^1 W^2 dy.
\eeq

\vskip0.3cm

$\bullet$ {\bf Change of variable on $\mathcal{D}$:} 
We first use Young's inequality: $-2AB\le \delta_1 A^2 +(4/\delta_1)B^2$ to get the following inequality: For any $A, B \in \bbr$,
\[
-(A+B)^2 \le -(1-\delta_1) A^2 +\frac{4}{\delta_1} B^2.
\]
Since
\[
\mathcal{D} = \int_\bbr a \frac{1}{\eta''(u)} \left| \eta''(u) \partial_\xi (u-S) + (\eta''(u)-\eta''(S)) S' \right|^2 d\xi,
\]
we use the above inequality to have
\[
-\mathcal{D} \le \underbrace{-(1-\delta_1)  \int_\bbr   \frac{1}{\eta''(u)} \left| \eta''(u) \partial_\xi (u-S) \right|^2 d\xi }_{=:\mathcal{D}_M} + \underbrace{\frac{4}{\delta_1} \int_\bbr   \frac{a^2}{\eta''(u)} |S'|^2 \left|\eta''(u)-\eta''(S) \right|^2 d\xi }_{=:\mathcal{D}_m}.
\]
Notice that  \eqref{us1}, \eqref{us2} and $\eta''\ge\alpha$ imply
\[
\mathcal{D}_m \le \frac{C}{\delta_1} \int_\bbr  |S'|^2 \left|u-S\right|^2 d\xi.
\]
Then using $|S'|\le\eps^2$ (by \eqref{tail}) together with \eqref{ach} and $\eps<\delta_1$, we have
\[
\mathcal{D}_m \le C \eps^2  \int_0^1 w^2 dy = C\frac{\eps^4}{\lambda^2}  \int_0^1 W^2 dy,
\]
which gives
\beq\label{minorD}
 \frac{2}{\eta''(u_-)f''(u_-)}\frac{\lambda^2}{\eps^3} \mathcal{D}_m \le C\delta_1  \int_0^1 W^2 dy.
\eeq
We now handle the main part:
\[
\mathcal{D}_M = -(1-\delta_1)  \int_\bbr \eta''(u)   \left| \partial_\xi (u-S) \right|^2 d\xi,
\]
which is rewritten into
\[
\mathcal{D}_M = -(1-\delta_1)  \int_0^1 \eta''(u) |\partial_y w|^2 \Big(\frac{dy}{d\xi}\Big) dy.
\]
Using the following Lemma \ref{lem:D}, and the fact that for any $u$ satisfying \eqref{us1} and \eqref{us2}:
\[
|\eta''(u)-\eta''(u_-)|\le |\eta''(u)-\eta''(S)|+|\eta''(S)-\eta''(u_-)| \le C\delta_1,
\]
we have
\[
\mathcal{D}_M  \le -  \eps \frac{\eta''(u_-)f''(u_-)}{2}  (1- C\delta_1)  \int_0^1 y(1-y)|\partial_y w|^2 dy.
\]
Thus,
\[
 \frac{2}{\eta''(u_-)f''(u_-)}\frac{\lambda^2}{\eps^3} \mathcal{D}_M \le -  (1- C\delta_1)  \int_0^1 y(1-y)|\partial_y W|^2 dy.
\]
Therefore, this and \eqref{minorD} yields
\beq\label{sum-D}
-  \frac{2}{\eta''(u_-)f''(u_-)}\frac{\lambda^2}{\eps^3}\mathcal{D} \le -  (1- C\delta_1)  \int_0^1 y(1-y)|\partial_y W|^2 dy +C\delta_1  \int_0^1 W^2 dy.
\eeq

\begin{lemma}\label{lem:D}
There exists a constant $C>0$ such that 
for any $\eps<\delta_1$, and any $y\in [0,1]$,
$$
\left|\frac{1}{y(1-y)}\frac{dy}{d\xi}-\eps \frac{f''(u_-)}{2}\right|\leq C\eps^2.
$$
\end{lemma}
\begin{proof} 
We first recall that (by \eqref{w1w2})
\[
\frac{dy}{d\xi} =-\frac{S'}{\eps},
\]
and (by \eqref{shock-1})
\begin{align*}
\begin{aligned}
S' = - \sigma (S -u_-) + f(S) -f(u_-).
\end{aligned}
\end{align*}
Using \eqref{RH-con}, we have
\[
S' = - \frac{1}{u_- -u_+} \Big( (f(u_-)-f(u_+)) (S -u_-) + (f(u_-)-f(S)) (u_--u_+) \Big).
\]
Then using  
\[
y=\frac{u_--S}{\eps},\quad 1-y=\frac{S-u_+}{\eps} \quad (\mbox{by } u_--u_+=\eps),
\]
we find
\[
\frac{1}{y(1-y)} \frac{dy}{d\xi} = \frac{f(S)-f(u_-)}{S-u_-} -  \frac{f(S)-f(u_+)}{S-u_+}.
\]
Therefore, we have
\begin{align*}
\begin{aligned}
\left|\frac{1}{y(1-y)}\frac{dy}{d\xi}-\eps \frac{f''(u_-)}{2}\right|   &= \left|  \frac{f(S)-f(u_-)}{S-u_-} -  \frac{f(S)-f(u_+)}{S-u_+} - \frac{f''(u_-)}{2} (u_--u_+) \right|  \\
&\le J_1+J_2+J_3+J_4+J_5,
\end{aligned}
\end{align*}
where
\begin{align*}
\begin{aligned}
&J_1:= \left|  \frac{f(S)-f(u_-)}{S-u_-} - f'(u_-) -  \frac{f''(u_-)}{2} (S-u_-) \right|, \\
&J_2:=  \left|  -  \frac{f(S)-f(u_+)}{S-u_+}  + f'(u_+) + \frac{f''(u_+)}{2} (S-u_+) \right|, \\
&J_3:=  \left|  f'(u_-) - f'(u_+) - f''(u_-) (u_- -u_+) \right| ,\\
&J_4:=  \left|  \frac{f''(u_-)}{2} (S-u_-) -  \frac{f''(u_+)}{2} (S-u_+) + \frac{f''(u_+)}{2} (u_--u_+) \right| ,\\
&J_5:=  \frac{1}{2} \left|  (f''(u_-)-f''(u_+)) (u_--u_+) \right|.
\end{aligned}
\end{align*}
We use Taylor theorem to have
\[
J_1+J_2+J_3+J_5 \le C\eps^2.
\]
Likewise, since
\[
J_4=  \frac{1}{2} \left|  (f''(u_-)-f''(u_+)) (S-u_-) \right|,
\]
we have $J_4  \le C\eps^2$.\\
Hence we have the desired estimate.
\end{proof}

$\bullet$ {\bf Uniform bound of $\int_0^1 W^2dy $:} 
We use \eqref{assYp} and \eqref{YW} to have
\begin{align*}
\begin{aligned}
\int_0^1 W^2dy-2\Big|\int_0^1 Wdy\Big|&\leq\int_0^1 W^2dy+2\int_0^1 Wdy\\
&\leq\left| \frac{2}{\eta''(u_-)} \frac{\lambda}{\eps^2} Y +   \int_0^1 W^2 dy +2 \int_0^1  W dy \right|+\frac{2}{\eta''(u_-)} \frac{\lambda}{\eps^2}|Y|\\
&\leq C\delta_1 \Big(\int_0^1 W^2 dy + \int_0^1 |W| dy\Big)+\frac{2}{\alpha} K,
\end{aligned}
\end{align*}
where $K$ is the constant in the assumption \eqref{assYp}.\\
Then using $$\Big|\int_0^1 Wdy\Big|\leq \int_0^1 |W|dy\leq \frac{1}{8}\int_0^1 W^2dy+2,$$
and taking $\delta_1$ small enough, we have
\begin{align*}
\begin{aligned}
\int_0^1 W^2dy
&\leq 2\Big|\int_0^1 Wdy\Big|+C\delta_1 \Big(\int_0^1 W^2 dy + \int_0^1 |W| dy\Big)+\frac{2}{\alpha}  K\\
&\leq \frac{1}{2}\int_0^1 W^2dy+C
\end{aligned}
\end{align*}
Therefore there exists a positive constant $M$ depending on $K$ such that
\beq\label{controlW}
\int_0^1 W^2dy \le M.
\eeq

 \vskip0.3cm
 $\bullet$ {\bf Estimate on  $-|Y|^2$:} 
 As in \cite{Kang-V-NS17}, we use the following inequality: For any $a,b\in \bbr$,  
 $$
 -a^2\leq -\frac{b^2}{2}+|b-a|^2.
 $$
 Using this inequality with 
 $$
 a=-  \frac{2}{\eta''(u_-)}\frac{\lambda}{\eps^2} Y ,\qquad b=\int_0^1W^2\,dy+2\int_0^1 W\,dy,
 $$
 we find
\begin{align*}
\begin{aligned}
  -  \frac{2}{\eta''(u_-)f''(u_-)} \frac{\lambda^2}{\eps^3} \frac{|Y|^2}{\eps \delta_1}
 &=-\frac{\eta''(u_-)}{2\delta_1 f''(u_-)}\left|   \frac{2}{\eta''(u_-)} \frac{\lambda}{\eps^2} Y \right|^2 \\
 &\leq -\frac{\eta''(u_-)}{4\delta_1 f''(u_-)} \left| \int_0^1W^2\,dy+2\int_0^1 W\,dy\right|^2\\
  &\quad +\frac{\eta''(u_-)}{2\delta_1 f''(u_-)} \left|  \frac{2}{\eta''(u_-)}  \frac{\lambda}{\eps^2} Y + \int_0^1 W^2 dy + 2\int_0^1 W dy  \right|^2.
\end{aligned}
\end{align*}
 Then using \eqref{YW}, we have
\begin{align*}
\begin{aligned}
  -  \frac{2}{\eta''(u_-)f''(u_-)}\frac{\lambda^2}{\eps^3} \frac{|Y|^2}{\eps \delta_1} 
  &\leq -\frac{\eta''(u_-)}{2\delta_1 f''(u_-)} \left| \int_0^1W^2\,dy+2\int_0^1 W\,dy\right|^2 \\
  &\quad +C\delta_1 \left(\int_0^1W^2\,dy+\int_0^1|W|\,dy \right)^2.
\end{aligned}
\end{align*}
Since \eqref{controlW} yields
 $$
 \left(\int_0^1W^2\,dy+\int_0^1|W|\,dy \right)^2\leq \left(\int_0^1W^2\,dy+\sqrt{\int_0^1|W|^2\,dy }\right)^2\leq C\int_0^1W^2\,dy,
 $$
we have
\beq\label{Y^2}
 -  \frac{2}{\eta''(u_-)f''(u_-)}\frac{\lambda^2}{\eps^3} \frac{|Y|^2}{\eps \delta_1} 
  \leq -\frac{\eta''(u_-)}{2\delta_1 f''(u_-)} \left| \int_0^1W^2\,dy+2\int_0^1 W \,dy\right|^2  +C\delta_1 \int_0^1W^2\,dy.
\eeq

\vskip0.3cm

$\bullet$ {\bf Conclusion:} 
We first find that for any $\delta<\delta_1$, 
\begin{align*}
\begin{aligned}
&\mathcal{R}_{\eps,\delta}\le -\frac{1}{\eps\delta_1}|Y|^2 +\left(\mathcal{B}_1+\mathcal{B}_2+\mathcal{B}_3+\mathcal{B}_4-\mathcal{G}_0 \right) \\
&\qquad\qquad\qquad + \delta_1\frac{\eps}{\lambda}\Big( |\mathcal{B}_1+\mathcal{B}_2+\mathcal{B}_3+\mathcal{B}_4 | +\mathcal{G}_0\Big) -(1-\delta_1)\mathcal{D}\le 0.
\end{aligned}
\end{align*}
Multiplying \eqref{sum-D} by $(1-\delta_1)$, and combining it with \eqref{sum-bg}, \eqref{sum-a} and \eqref{Y^2} with putting $C_*:=\frac{2 f''(u_-)}{\eta''(u_-)}$, we find 
\begin{align*}
\begin{aligned}
& \frac{2}{\eta''(u_-)f''(u_-)}\frac{\lambda^2}{\eps^3} \mathcal{R}_{\eps,\delta}\\
&\quad\le-\frac{1}{C_*\delta_1}\left(\int_0^1W^2\,dy+2\int_0^1 W\,dy\right)^2+(1+C \delta_1)\int_0^1 W^2\,dy\\
&\qquad+\frac{2}{3}\int_0^1 W^3\,dy +C \delta_1\int_0^1 |W|^3\,dy  -(1-C \delta_1)\int_0^1 y(1-y)|\partial_y W|^2\,dy.
\end{aligned}
\end{align*}
Let $\delta^*$ be the constant in Lemma \ref{lem_poincare} corresponding to the constant $M$ of \eqref{controlW}. \\
Taking $\delta_1$ small enough such that $\max(C_*, C) \delta_1\le \delta^*$, we have
\begin{align*}
\begin{aligned}
& \frac{2}{\eta''(u_-)f''(u_-)}\frac{\lambda^2}{\eps^3} \mathcal{R}_{\eps,\delta}\\
&\quad\le-\frac{1}{\delta_*}\left(\int_0^1W^2\,dy+2\int_0^1 W\,dy\right)^2+(1+ \delta_*)\int_0^1 W^2\,dy\\
&\qquad+\frac{2}{3}\int_0^1 W^3\,dy + \delta_*\int_0^1 |W|^3\,dy  -(1- \delta_*)\int_0^1 y(1-y)|\partial_y W|^2\,dy =: R_{\delta_*}(W).
\end{aligned}
\end{align*}
Therefore, using $R_{\delta_*}(W)\leq0$ by Lemma \ref{lem_poincare}, we have  $\mathcal{R}_{\eps,\delta}\leq 0$.

\subsection{Truncation of the big values of $|\eta'(u)-\eta'(S)|$}\label{section-finale}
In order to use Proposition \ref{prop:near}, we need to show that the values for $|u|$ such that $|\eta'(u)-\eta'(S)| \geq\delta_1$ have a small effect. However, the value of $\delta_1$ is itself conditioned to the constant $K$ in Proposition \ref{prop:near}. Therefore, we need first to find a uniform bound on $Y$ which is not yet conditioned on the level of truncation $\delta_1$.\\

We consider a truncation on $|\eta'(u)-\eta'(S)|$ with a constant $r>0$. Later we will fix $r$ as $r=\delta_1$ where $\delta_1$ is the constant  in Proposition \ref{prop:near}.

For a given $r>0$, let $\psi_r$ be a continuous function defined by
\beq\label{psi}
 \psi_r(y):=\begin{cases}
&y \mbox{ if } |y|\leq r \\ 
&r \mbox{ if }  y>r \\
&-r \mbox{ if }  y<-r.
\end{cases}
\eeq
We then define the function $\bar u_r$ by
\beq\label{def-bar}
\eta'(\bar u_r)-\eta'(S) = \psi_r \left( \eta'(u)-\eta'(S) \right).
\eeq
Notice that once $r$ is fixed, $\bar u_r$ is uniquely determined since $\eta'$ is one to one. 

We first have the following lemma.
\begin{lemma}\label{lem_ey}
There exist constants $\delta_0$, $C, K>0$ such that 
for any $\eps, \lambda>0$ with $\delta_0^{-1}\eps<\lambda<\delta_0$, the following holds whenever $|Y(u)|\leq \eps^2$:
\beq\label{l1} 
\int_\bbr a' \eta(u|S)\,d\xi \leq C\frac{\eps^2}{\lambda},
\eeq 
and 
\beq\label{y-small}
 |Y(\bar u_r)|\leq K \frac{\varepsilon^2}{\lambda}\quad \mbox{ for any }  r>0 .
\eeq  
\end{lemma}

\begin{proof}
$\bullet$ {\it Proof of \eqref{l1}:}
By the definition of $Y$, and $S'=-\frac{\eps}{\lambda}a'$, we have
\begin{align*}
\begin{aligned}
\int_\bbr a'\eta(u|S) d\xi &= -Y +\int_\bbr a\, \eta''(S) S' (u-S) d\xi \\
&\le \eps^2 + C \frac{\eps}{\lambda} \int_\bbr a' |u-S| d\xi\\
&\le \eps^2 + C \frac{\eps}{\lambda} \sqrt{\int_\bbr a' d\xi}\sqrt{\int_\bbr a' |u-S|^2 d\xi}.
\end{aligned}
\end{align*}
Then using \eqref{e-est0} and Young's inequality, we have
\begin{align*}
\begin{aligned}
\int_\bbr a'\eta(u|S) d\xi &\le \eps^2 + C \frac{\eps}{\lambda} \sqrt{\int_\bbr a' d\xi}\sqrt{\int_\bbr a' |u-S|^2 d\xi}\\
&\le \eps^2 + C \frac{\eps}{\sqrt\lambda}\sqrt{\frac{2}{\alpha}\int_\bbr a' \eta(u|S) d\xi}\\
&\le C\frac{\eps^2}{\lambda} + \frac{1}{2}\int_\bbr a' \eta(u|S) d\xi.
\end{aligned}
\end{align*}
Hence we have \eqref{l1}.

$\bullet$ {\it Proof of \eqref{y-small}:} 
Using the same estimates as above, we have
\begin{align*}
\begin{aligned}
|Y(\bar u_r)| &\le \int_\bbr a'\eta(\bar u_r|S) d\xi + C \frac{\eps}{\lambda} \int_\bbr a' |\bar u_r-S| d\xi\\
&\le  \int_\bbr a'\eta(\bar u_r|S) d\xi  + C \frac{\eps}{\sqrt\lambda}\sqrt{\int_\bbr a' \eta(\bar u_r|S) d\xi}.
\end{aligned}
\end{align*}
Since \eqref{def-bar} together with $\eta'>0$ implies either $S\le \bar u_r \le u$ or $u\le \bar u_r \le S$, it follows from \eqref{e-mono} that
\[
\eta(u|S) \ge \eta(\bar u_r |S).
\]
Therefore, using \eqref{l1}, there exists a constant $K>0$ such that
\[
|Y(\bar u_r)|\le  \int_\bbr a'\eta(u|S) d\xi  + C \frac{\eps}{\sqrt\lambda}\sqrt{\int_\bbr a' \eta(u|S) d\xi} \leq K \frac{\varepsilon^2}{\lambda}.
\]
\end{proof}

We now fix the constant $\delta_1 \le \min(1, 2|u_-|)$ of Proposition \ref{prop:near} associated to the constant $K$ of Lemma \ref{lem_ey}. From now on, we set 
$$
\bar u:=\bar u_{\delta_1}.
$$
Then it follows from Lemma \ref{lem_ey} and \eqref{def-bar} that
\begin{equation}\label{YC2}
|Y(\bar u)|\leq K \frac{\eps^2}{\lambda},\quad \left| \eta'(\bar u)-\eta'(S) \right| \le \delta_1,
\end{equation}
and therefore, it follows from Proposition \ref{prop:near} that
\begin{align}
\begin{aligned}\label{final-R}
&\mathcal{R}_{\eps,\delta_1}(\bar u)=-\frac{1}{\eps\delta_1}|Y(\bar  u)|^2 +\mathcal{B}_1(\bar  u)+\mathcal{B}_2(\bar  u)+\mathcal{B}_3( \bar u)+\mathcal{B}_4(\bar  u) \\
&\qquad\qquad\quad + \delta_1\frac{\eps}{\lambda}|\mathcal{B}_1(\bar  u)+\mathcal{B}_2(\bar u)+\mathcal{B}_3(\bar u)+\mathcal{B}_4(\bar u)| -\left(1-\delta_1\frac{\eps}{\lambda}\right)\mathcal{G}_0(\bar u)-(1-\delta_1)\mathcal{D}(\bar u)\le 0.
\end{aligned}
\end{align}

We recall that the functional  $\mathcal{G}$ in \eqref{badgood} consists of the two good terms $\mathcal{G}_0$ and $ \mathcal{D}$ in \eqref{badgood-n}, that is $\mathcal{G}=\mathcal{G}_0+\mathcal{D}$.\\
Note that it follows from \eqref{e-mono} that
\begin{equation}\label{eq_G}
 \mathcal{G}_0(u)- \mathcal{G}_0(\bar u)=\sigma \int_\bbr a' \left( \eta (u|S)-\eta(\bar u |S) \right)\,d\xi\geq 0,
\end{equation}
which together with  \eqref{l1} yields
\beq\label{l2}
0\leq \mathcal{G}_0(u)- \mathcal{G}_0(\bar u) \leq  C {\intr}a' \eta (u|S) d\xi \leq C\frac{\eps^2}{\lambda}.
\eeq
Also note that since $\eta'(\bar u)-\eta'(S)$ is constant for $u$ satisfying either $\eta'(u)-\eta'(S)<-\delta_1$ or $\eta'(u)-\eta'(S)>\delta_1$ (by \eqref{def-bar}), we find
$$
\mathcal{D}(\bar u)= \int_\bbr a  \mu(u) |\partial_\xi (\eta'(u)-\eta'(S))|^2 {\mathbf 1}_{\{|\eta'(u)-\eta'(S)) |\leq\delta_1\}} d\xi,
$$
and
\begin{equation}\label{bar-eq}
\begin{array}{rl}
|\eta'(u)-\eta'(\bar u)|=& |(\eta'(u)-\eta'(S))+(\eta'(S)-\eta'(\bar u))|\\[0.2cm]
=&|(\psi_{\delta_1}-I)(\eta'(u)-\eta'(S))|\\[0.2cm]
=& (|\eta'(u)-\eta'(S)|-\delta_1)_+,
\end{array}
\end{equation}
and therefore,
\begin{align}
\begin{aligned}\label{D_bar}
\mathcal{D}(u)&=\int_\bbr a \mu(u) |\partial_\xi (\eta'(u)-\eta'(S))|^2 d\xi\\
&=\int_\bbr a \mu(u) |\partial_\xi (\eta'(u)-\eta'(S))|^2( {\mathbf 1}_{\{|\eta'(u)-\eta'(S) |\leq\delta_1\}} + {\mathbf 1}_{\{|\eta'(u)-\eta'(S) |>\delta_1\}} )d\xi\\
&=\mathcal{D}(\bar u)+\int_\bbr a  \mu(u)  |\partial_\xi (\eta'(u)-\eta'(\bar u))|^2 d\xi\\
&\ge \int_\bbr a  \mu(u)  |\partial_\xi (\eta'(u)-\eta'(\bar u))|^2 d\xi,
\end{aligned}
\end{align}
which also yields
\begin{equation}\label{eq_D}
\mathcal{D}(u)-\mathcal{D}(\bar u)=\int_\bbr a  \mu(u)  |\partial_\xi (\eta'(u)-\eta'(\bar u))|^2 d\xi \geq0.
\end{equation}

For bad terms of $\mathcal{B}$ in \eqref{badgood}, we will use the following notations :
\beq\label{bad0}
\mathcal{B}(u)=\sum_{i=1}^8\mathcal{B}_i(u),
\eeq
where $\mathcal{B}_1(u), \mathcal{B}_2(u), \mathcal{B}_3(u), \mathcal{B}_4(u)$ are defined by \eqref{badgood-n}, and
\begin{align}
\begin{aligned}\label{newbad}
&\mathcal{B}_5(u):= - \int_\bbr a' \mu(u) \left( \eta'(u) - \eta'(S) \right)\partial_\xi  \left( \eta'(u) - \eta'(S) \right) d\xi\\
&\mathcal{B}_6(u):= - \int_\bbr a' \left( \eta'(u) - \eta'(S) \right) \left( \mu(u) - \mu(S) \right) \eta''(S)S' d\xi\\
&\mathcal{B}_7(u):= - \int_\bbr a \partial_\xi  \left( \eta'(u) - \eta'(S) \right) \left( \mu(u) - \mu(S) \right) \eta''(S)S' d\xi \\
&\mathcal{B}_8(u):= \int_\bbr a S''  (\eta')(u|S) d\xi,\\
\end{aligned}
\end{align}

We now state the following proposition.

\begin{proposition}\label{prop_out}
There exist constants $\delta_0, C, C^*>0$ (in particular, $C$ depends on the constant $\delta_1$ of Proposition \ref{prop:near}) such that for any $\delta_0^{-1}\eps<\lambda<\delta_0$, the following statements hold.
\begin{itemize}
\item[1.] For any $u$ such that $|Y(u)|\leq \eps^2$,
\begin{eqnarray}
\label{n1}
&&\sum_{i=1}^4|\mathcal{B}_i(u)-\mathcal{B}_i(\bar u) | \le C\sqrt{\frac{\eps}{\lambda}} \mathcal{D}(u),\\
\label{n2}
&&\sum_{i=5}^8 |\mathcal{B}_i(u)| \le \delta_0^{1/3}  \mathcal{D}(u) + C\delta_0\frac{\eps}{\lambda} \mathcal{G}_0(\bar u),\\
\label{n3}
&&|\mathcal{B}(U)| \le  \delta_0^{1/4} \mathcal{D}(u) + C^* \frac{\eps^2}{\lambda}.
\end{eqnarray}
\item[2.] For any $u$ such that $|Y(u)|\leq \eps^2$ and $\mathcal{D}(u)\leq \frac{C^*}{4}\frac{\eps^2}{\lambda}$,
\beq\label{m1}
|Y(u)-Y(\bar u)|^2 \le C \frac{\eps^2}{\lambda} \frac{\eps}{\lambda}\mathcal{D}(u).
\eeq
\end{itemize}
\end{proposition}

To prove Proposition \ref{prop_out}, we first show the following estimates.

\begin{lemma}\label{lemma_oublie}
Under the same assumption as in Proposition \ref{prop_out}, for any $u$ such that $|Y(u)|\leq \eps^2$, the following holds:
\begin{eqnarray}
\label{pw}
&& \left|(\eta' (u)- \eta'(\bar u))(\xi)\right|   \le C\sqrt{|\xi|+\frac{1}{\eps}}\sqrt{\mathcal{D}(u)}\qquad \mbox{for all } \xi\in\bbr,\\
\label{l3}
&&  \int_\bbr a' \Big( |\eta' (u)- \eta'(\bar u)|^2 + |\eta' (u)- \eta'(\bar u)|  \Big) \,d\xi \leq C\sqrt{\frac{\eps}{\lambda}}\mathcal{D}(u),\\
\label{l4}
&&  \int_\bbr a' |u-\bar u| \,d\xi \leq C\sqrt{\frac{\eps}{\lambda}}\mathcal{D}(u).
\end{eqnarray}
\end{lemma}
\begin{proof}
$\bullet$ {\it Proof of \eqref{pw}:}
First, using \eqref{l1} and \eqref{low-S} together with $a'=(\lambda/\eps) |S'|$, we get
\begin{eqnarray*}
2\eps\int_{-1/\eps}^{1/\eps} \eta (u|S)\, d\xi&\leq& \frac{2\eps}{\inf_{[-1/\eps,1/\eps]} a'}\int_\bbr a' \eta (u|S)\,d\xi\\
&\leq& C \frac{\eps}{\lambda\eps}\frac{\eps^2}{\lambda}=C\left(\frac{\eps}{\lambda}\right)^2.
\end{eqnarray*}
Then there exists $\xi_0\in [-1/\eps,1/\eps]$ such that 
\[
\eta(u(\xi_0)|S(\xi_0))\leq C\left(\frac{\eps}{\lambda}\right)^2.
\]
Thus using \eqref{e-est3}, we have 
$$
|(\eta'(u)-\eta'(S))(\xi_0)|\leq C\frac{\eps}{\lambda}.
$$
Therefore, if $\deo$ is small enough such that $C\eps/\lambda \le C\deo \leq \delta_1/2$, then the definition of $\bar u$ implies
$$
(\eta'(u)-\eta'(\bar u))(\xi_0)=0.
$$
Hence using the assumption $\eta''\ge\alpha$ together with \eqref{D_bar}, we find that for any $\xi\in \bbr$,
\begin{align*}
\begin{aligned}
|(\eta'(u)-\eta'(\bar u)) (\xi)|&=\left|\int_{\xi_0}^\xi \partial_\zeta(\eta'(u)-\eta'(\bar u)) \,d\zeta\right|\\
&\le \frac{1}{\sqrt\alpha} \sqrt{|\xi-\xi_0|} \sqrt{\int_\bbr a \mu(u) |\partial_\zeta (\eta'(u)-\eta'(\bar u)) |^2 \,d\zeta }\\
&\le  C\sqrt{|\xi|+\frac{1}{\eps}}\sqrt{\mathcal{D}(u)}.
\end{aligned}
\end{align*}

$\bullet$ {\it Proof of \eqref{l3}:}
We first notice that since $(y-\delta_3/2)_+\geq \delta_3/2$ whenever $(y-\delta_3)_+>0$, we have
\beq\label{y-identity}
(y-\delta_3)_+\leq (y-\delta_3/2)_+{\mathbf 1}_{\{y-\delta_3>0\}}\leq (y-\delta_3/2)_+\left(\frac{(y-\delta_3/2)_+}{\delta_3/2}\right)\leq \frac{2}{\delta_3} (y-\delta_3/2)_+^2.
\eeq
Therefore, to show \eqref{l4}, it is enough to handle the quadratic part with $\bar v$ defined with $\delta_3/2$ instead of $\delta_3$. We will keep the notation $\bar v$ for this case below.\\

We split the quadratic part into two parts:
\begin{align*}
\begin{aligned}
\int_\bbr a' |\eta' (u)- \eta'(\bar u)|^2 \,d\xi &\le \underbrace{\int_{|\xi|\le\frac{1}{\eps}\sqrt{\frac{\lambda}{\eps}}} a' |\eta' (u)- \eta'(\bar u)|^2  \,d\xi }_{=:J_1}  + \underbrace{\int_{|\xi|\geq\frac{1}{\eps}\sqrt{\frac{\lambda}{\eps}}} a' |\eta' (u)- \eta'(\bar u)|^2 \,d\xi }_{=:J_2}.
\end{aligned}
\end{align*}
To control $J_1$, we first observe that the definition of $\bar u$ implies 
\[
|(\eta'(u)-\eta'(S))(\xi)|> \delta_1, \quad\mbox{whenever $\xi$ satisfying $|(\eta'(u)-\eta'(\bar u))(\xi)|>0$}. 
\]
Since
\beq\label{S-bdd}
|S| \le 2|u_-| \quad\mbox{by  $|S-u_-|\le\eps<\delta_0^2< |u_-|  $~ for $\delta_0$ small enough,} 
\eeq
\beq\label{use1}
\mbox{we use the hypothesis (i) of ($\mathcal{H}$2) with holding $\theta=2|u_-|$ to find}
\eeq
\[
\eta(u|S) \ge C\delta_1^2 , \quad\mbox{whenever $\xi$ satisfying $|(\eta'(u)-\eta'(\bar u))(\xi)|>0$},
\]
which gives
\begin{equation}\label{beta2}
{\mathbf 1}_{\{|\eta'(u)-\eta'(\bar u)|>0\}}\leq C \delta_1^{-2} \eta(u|S).
\end{equation}
We now use \eqref{pw}, \eqref{beta2} and \eqref{l1} to estimate
\begin{align*}
\begin{aligned}
J_1&\leq\left(\sup_{\left[-\frac{1}{\eps}\sqrt{\frac{\lambda}{\eps}} ,\frac{1}{\eps}\sqrt{\frac{\lambda}{\eps}} \right]} |\eta' (u)- \eta'(\bar u)|^2 \right)  \int_{|\xi|\le\frac{1}{\eps}\sqrt{\frac{\lambda}{\eps}}} a'  {\mathbf 1}_{\{|\eta'(u)-\eta'(\bar u)|>0\}} d\xi \\
& \leq C\frac{1}{\eps}\sqrt{\frac{\lambda}{\eps}} \mathcal{D}(u)   \int_\bbr a'  \eta(u|S) \,d\xi \leq C\sqrt{\frac{\eps}{\lambda}} \mathcal{D}(u),
\end{aligned}
\end{align*}
where the constant $C$ depends on $\delta_1$.\\

Using \eqref{pw} and \eqref{tail} with $a' =(\eps/\lambda) S'$, we have
\begin{align*}
\begin{aligned}
J_2&\leq C\mathcal{D}(u)   \int_{|\xi|\geq\frac{1}{\eps}\sqrt{\frac{\lambda}{\eps}}} a' \left(|\xi|+\frac{1}{\eps}\right)\,d\xi \leq C\mathcal{D}(u)   \int_{|\xi|\geq\frac{1}{\eps}\sqrt{\frac{\lambda}{\eps}}} a'  |\xi| \,d\xi\\
& \le C\mathcal{D}(u)  \eps\lambda\int_{|\xi|\geq\frac{1}{\eps}\sqrt{\frac{\lambda}{\eps}}}e^{-c\eps|\xi|} |\xi| \,d\xi \le C\mathcal{D}(u) \frac{\lambda}{\eps} \int_{|\xi|\geq \sqrt{\frac{\lambda}{\eps}}} |\xi| e^{-c|\xi|}d\xi.
\end{aligned}
\end{align*}
Then, taking $\deo$ small enough such that $|\xi|\leq e^{(c/2)|\xi|}$ for $\xi\geq \sqrt{\lambda/\eps}$ where $\eps/\lambda\leq \deo$, we have
$$
\frac{\lambda}{\eps} \int_{|\xi|\geq \sqrt{\frac{\lambda}{\eps}}} |\xi| e^{-c|\xi|}d\xi \le \frac{\lambda}{\eps} \int_{|\xi|\geq \sqrt{\frac{\lambda}{\eps}}} e^{-\frac{c}{2}|\xi|}d\xi=\frac{2\lambda}{c\eps}e^{-\frac{c}{2}\sqrt{\frac{\lambda}{\eps}}}\leq \sqrt{\frac{\eps}{\lambda}}.
$$
Hence we have 
\[
\int_\bbr a' |\eta' (u)- \eta'(\bar u)|^2 \,d\xi \le C\sqrt{\frac{\eps}{\lambda}} \mathcal{D}(u).
\]
We now recall $\bar u=\bar u_{\delta_3/2}$ in the above estimate, as mentioned before. Therefore we use \eqref{bar-eq} to have
\begin{align*}
\begin{aligned}
\int_\bbr a' |\eta' (u)- \eta'(\bar u_{\delta_3})|^2\,d\xi &=\int_\bbr a'(|\eta' (u)- \eta'(S)|-\delta_3)_+^2\,d\xi \\
&\le \int_\bbr a'(|\eta' (u)- \eta'(S)| -\delta_{\delta_3/2})_+^2\,d\xi\\
&=\int_\bbr a' |\eta' (u)- \eta'(\bar u_{\delta_3/2})|^2\,d\xi \leq C\sqrt{\frac{\eps}{\lambda}} \mathcal{D}(u).
\end{aligned}
\end{align*}
Likewise, using \eqref{bar-eq} and \eqref{y-identity} with $y:=|p(v)-p(\tilde v_\eps)|$, we have
\[
\int_\bbr a'|\eta' (u)- \eta'(\bar u_{\delta_3})|\,d\xi \le \frac{2}{\delta_3}\int_\bbr a' |\eta' (u)- \eta'(\bar u_{\delta_3/2})|^2 \,d\xi \leq C\sqrt{\frac{\eps}{\lambda}} \mathcal{D}(u).
\]
Hence we have \eqref{l3}.\\

$\bullet$ {\it Proof of \eqref{l4}:}
Since the hypothesis $\eta''\ge\alpha$ implies
\[
|\eta'(u)-\eta'(\bar u)| \ge \alpha |u-\bar u |,
\]
it follows from \eqref{l4} that
\[
\int_\bbr a'  | u- \bar u|  \,d\xi  \le \frac{1}{\alpha}\int_\bbr a'  |\eta' (u)- \eta'(\bar u)|  \,d\xi \leq C\sqrt{\frac{\eps}{\lambda}}\mathcal{D}(u).
\]
\end{proof}

\subsubsection{\bf Proof of Proposition \ref{prop_out}} We first show the estimates of \eqref{n1}.\\

$\bullet$ {\it Proof of \eqref{n1}} :
Recall the functional $\mathcal{B}_1$ in \eqref{badgood-n}. Since $F'=-\eta'' f$ is continuous and $|S|\le 2|u_-|$, the definition of the relative functional implies 
\beq\label{rel-F}
|F(u|S)-F(\bar u|S)| \le |F(u)-F(\bar u)| + C|u-\bar u|.
\eeq
Thus, we have
\[
|\mathcal{B}_1(u)-\mathcal{B}_1(\bar u) | \le  \underbrace{\int_\bbr a' |F(u)-F(\bar u)| d\xi }_{=:J}+   \int_\bbr a' |u-\bar u| d\xi .
\]
To control $J$, we first observe that using \eqref{S-bdd} and $|\eta'(\bar u)-\eta'(S)|\le \delta_1$, we have
\beq\label{u-bdd}
|\bar u|\le |\bar u- S|+|S|\le \frac{|\eta'(\bar u)-\eta'(S)|}{\alpha} +2 |u_-| \le \frac{\delta_1}{\alpha} +2|u_-| \le 3|u_-|.
\eeq
\beq\label{use2}
\mbox{Then it follows from the hypothesis (iv) of ($\mathcal{H}$2) with holding $\theta=   3|u_-|$ that}
\eeq
\[
 |F(u)-F(\bar u)| \le C \Big( |\eta'(u)-\eta'(\bar u)| {\mathbf 1}_{\{ |u| \le 2\theta \}}  +|\eta'(u)-\eta'(\bar u)|^2 {\mathbf 1}_{\{|u| > 2\theta \}}\Big).
\]
Therefore we use \eqref{l3} to have
\[
J \leq  C \int_\bbr a'  \Big( |\eta'(u)-\eta'(\bar u)| +|\eta'(u)-\eta'(\bar u)|^2 \Big) d\xi \le C\sqrt{\frac{\eps}{\lambda}} \mathcal{D}(u).
\]
Moreover using \eqref{l4}, we have
\[
|\mathcal{B}_1(u)-\mathcal{B}_1(\bar u) | \le C\sqrt{\frac{\eps}{\lambda}} \mathcal{D}(u).
\]
Likewise, for $\mathcal{B}_3$ and $\mathcal{B}_4$, we use the same reason as in \eqref{rel-F} to have
\begin{align*}
\begin{aligned}
&|(\eta')(u|S)-(\eta')(\bar u|S)| \le |\eta'(u)-\eta'(\bar u)| + C|u-\bar u|,\\
&|f(u|S)-f(\bar u|S)| \le |f(u)-f(\bar u)| + C|u-\bar u|.
\end{aligned}
\end{align*}
Then using $|f(S)|\le C$ and $|\eta''(S)|\le C$ together with $S' =(\eps/\lambda) a'$, we have
\begin{align*}
\begin{aligned}
&|\mathcal{B}_3(u)-\mathcal{B}_3(\bar u) | \le  C \Big(\int_\bbr a' |\eta'(u)-\eta'(\bar u)| d\xi +   \int_\bbr a' |u-\bar u| d\xi \Big) ,\\
&|\mathcal{B}_4(u)-\mathcal{B}_4(\bar u) | \le  C \frac{\eps}{\lambda} \Big(\int_\bbr a' |f(u)-f(\bar u)| d\xi +  C \int_\bbr a' |u-\bar u| d\xi  \Big).
\end{aligned}
\end{align*}
Using \eqref{l3} and \eqref{l4}, we have
\[
|\mathcal{B}_3(u)-\mathcal{B}_3(\bar u) | \le C\sqrt{\frac{\eps}{\lambda}} \mathcal{D}(u).
\]
\beq\label{use3}
\mbox{Using the hypothesis (ii) of ($\mathcal{H}$2) with holding $\theta= 3|u_-|$ (by \eqref{u-bdd}), we have}
\eeq
\[
|\mathcal{B}_4(u)-\mathcal{B}_4(\bar u) | \le  C \frac{\eps}{\lambda} \Big(\int_\bbr a' |\eta'(u)-\eta'(\bar u)| d\xi +   \int_\bbr a' |u-\bar u| d\xi \Big)\le C\sqrt{\frac{\eps}{\lambda}} \mathcal{D}(u).
\]
To estimate $|\mathcal{B}_2(u)-\mathcal{B}_2(\bar u)|$, we first separate it into two parts:
\begin{align*}
\begin{aligned}
|\mathcal{B}_2(u)-\mathcal{B}_2(\bar u)|  & = \Big| \int_\bbr a'  \left( \eta'(u) - \eta'(S) \right)\left( f(u) - f(S) \right) -\left( \eta'(\bar u) - \eta'(S) \right)\left( f(\bar u) - f(S) \right)  d\xi \Big|\\
&= \Big| \int_\bbr a'  \left( \eta'(u) - \eta'(\bar u) \right)\left( f(u) - f(S) \right) +\left( \eta'(\bar u) - \eta'(S) \right)\left( f(u) - f(\bar u) \right)  d\xi \Big|\\
&\le  \int_\bbr a' \left| \eta'(u) - \eta'(\bar u) \right|\left| f(u) - f(S) \right| d\xi \\
&\quad + \int_\bbr a' \left| \eta'(\bar u) - \eta'(S) \right|\left| f(u) - f(\bar u) \right|  d\xi =: J_1 +J_2.
\end{aligned}
\end{align*}
\beq\label{use4}
\mbox{Since the hypothesis (ii) of ($\mathcal{H}$2) with holding $\theta=2|u_-|$ (by \eqref{S-bdd}) yields}
\eeq
\[
|f(u)-f(S)|\le C | \eta'(u) - \eta'(S) |,
\]
we use the definition of $\bar u$ and \eqref{l3} to have
\begin{align*}
\begin{aligned}
J_1  & \le  \int_\bbr a' \left| \eta'(u) - \eta'(\bar u) \right|\left| \eta'(u) - \eta'(S) \right| d\xi \\
& \le  \int_\bbr a' \left| \eta'(u) - \eta'(\bar u) \right| \Big( \left| \eta'(u) - \eta'(\bar u) \right| +\left| \eta'(\bar u) - \eta'(S) \right| \Big)d\xi \\
& \le  \int_\bbr a' \Big(\left| \eta'(u) - \eta'(\bar u) \right|^2 + \delta_1 \left| \eta'(u) - \eta'(\bar u) \right| \Big) d\xi  \le C\sqrt{\frac{\eps}{\lambda}} \mathcal{D}(u).
\end{aligned}
\end{align*}
\beq\label{use5}
\mbox{Likewise, using (ii) of ($\mathcal{H}$2) with holding $\theta= 3 |u_-|$, we have}
\eeq
\begin{align*}
\begin{aligned}
J_2  & \le  \int_\bbr a' \left| \eta'(\bar u) - \eta'(S) \right|\left| \eta'(u) - \eta'(\bar u) \right| d\xi \\
& \le  \int_\bbr a' \delta_1 \left| \eta'(u) - \eta'(\bar u) \right| d\xi \le C\sqrt{\frac{\eps}{\lambda}} \mathcal{D}(u).
\end{aligned}
\end{align*}
Therefore we have
\[
|\mathcal{B}_2(u)-\mathcal{B}_2(\bar u)|   \le  C\sqrt{\frac{\eps}{\lambda}} \mathcal{D}(u).
\]

$\bullet$ {\it Proof of \eqref{n2}} :
Recall the functionals $\mathcal{B}_i, 5\le i\le 8$.\\
Using Young's inequality together with $\mu\le 1/\alpha$ and $a'\le \eps\lambda <\eps\delta_0$, we first have
\begin{align*}
\begin{aligned}
|\mathcal{B}_5(u)| &\le \delta_0 \int_\bbr a \mu(u)\left| \partial_\xi  \left( \eta'(u) - \eta'(S) \right) \right|^2 d\xi +\frac{C}{\delta_0} \int_\bbr |a'|^2  \left| \eta'(u) - \eta'(S) \right|^2 d\xi \\
&\le \delta_0 \mathcal{D}(u) + \underbrace{C \eps \int_\bbr a'  \left| \eta'(u) - \eta'(S) \right|^2 d\xi}_{=:B_{51}(u)}.
\end{aligned}
\end{align*}
We separate the remaining term $B_{51}(u)$ into
\[
|B_{51}(u)|\le |B_{51}(u)-B_{51}(\bar u)|+ |B_{51}(\bar u)|.
\] 
Using $|\eta'(\bar u)-\eta'(S)|\leq \delta_1$ and  \eqref{l3}, we have
\begin{align*}
\begin{aligned}
 |B_{51}(u)-B_{51}(\bar u)| &\le C  \int_\bbr a'  \left|  \left| \eta'(u) - \eta'(S) \right|^2-  \left| \eta'(\bar u) - \eta'(S) \right|^2 \right| d\xi\\
 &= C  \int_\bbr a'  \left| \eta'(u) - \eta'(\bar u) \right| \left|\eta'(u)+ \eta'(\bar u) - 2\eta'(S) \right| d\xi \\
 &\le C  \int_\bbr a'  \left| \eta'(u) - \eta'(\bar u) \right| \Big(\left|\eta'(u)- \eta'(\bar u) \right| + 2 \left| \eta'(\bar u)-\eta'(S) \right|\Big) d\xi \\
  &\le C  \int_\bbr a'  \Big(\left|\eta'(u)- \eta'(\bar u) \right|^2 + 2\delta_1 \left|\eta'(u)- \eta'(\bar u) \right| \Big) d\xi 
   \le  C\sqrt{\frac{\eps}{\lambda}} \mathcal{D}(u).
\end{aligned}
\end{align*}
\beq\label{use6}
\mbox{Using the hypothesis (i) of ($\mathcal{H}$2) with holding $\theta= 2|u_-|$, by \eqref{S-bdd} and \eqref{u-bdd}, we have}
\eeq
\begin{align*}
\begin{aligned}
 |B_{51}(\bar u)| \le C \eps \int_\bbr a'  \left| \eta'(\bar u) - \eta'(S) \right|^2 d\xi \le C \eps \int_\bbr a'   \eta(\bar u|S) d\xi.
\end{aligned}
\end{align*}
Since $\eps<\delta_0 \eps/\lambda$, we have
\[
|B_{51}(\bar u)| \le C\delta_0\frac{\eps}{\lambda} \mathcal{G}_0(\bar u).
\]
Therefore, for $\eps/\lambda <\delta_0\ll1$,
\beq\label{B5}
|\mathcal{B}_5(u)| \le C\delta_0^{1/2}  \mathcal{D}(u) + C\delta_0\frac{\eps}{\lambda} \mathcal{G}_0(\bar u).
\eeq

To estimate $|\mathcal{B}_6(u)| $, we separate it into
\[
|\mathcal{B}_6(u)| \le |\mathcal{B}_6(u)-\mathcal{B}_6(\bar u)|+ |\mathcal{B}_6(\bar u)|.
\] 
\beq\label{use7}
\mbox{Using the hypotheses $\eta''\ge\alpha$, and (i), (iii) of ($\mathcal{H}$2) with holding $\theta= 2|u_-|$,}
\eeq
by \eqref{S-bdd} and \eqref{u-bdd}, we have
\begin{align*}
\begin{aligned}
|\mathcal{B}_6(\bar u)| &\le C \eps^2 \int_\bbr a'  \left| \eta'(\bar u) - \eta'(S) \right| \left|\mu(\bar u) - \mu(S) \right|  d\xi\\
&\le C \eps^2 \int_\bbr a'  \left| \eta'(\bar u) - \eta'(S) \right| \left|\eta''(\bar u) - \eta''(S) \right|  d\xi\\
&\le C \eps^2 \int_\bbr a'  \left| \eta'(\bar u) - \eta'(S) \right|^2 d\xi\\
&\le C \delta_0\frac{\eps}{\lambda}  \int_\bbr a'   \eta(\bar u|S) d\xi \le C\delta_0\frac{\eps}{\lambda} \mathcal{G}_0(\bar u).
\end{aligned}
\end{align*}
Likewise, we have
\begin{align*}
\begin{aligned}
|\mathcal{B}_6(u)-\mathcal{B}_6(\bar u)| &\le C  \int_\bbr a'  \left|  \left( \eta'(u) - \eta'(S) \right)\left( \mu(u) - \mu(S) \right)- \left( \eta'(\bar u) - \eta'(S) \right)\left( \mu(\bar u) - \mu(S) \right) \right|  d\xi\\
&= C  \int_\bbr a'  \left|  \left( \eta'(u) - \eta'(\bar u) \right)\left( \mu(u) - \mu(S) \right)- \left( \eta'(\bar u) - \eta'(S) \right)\left( \mu(u) - \mu(\bar u) \right) \right|  d\xi\\
&\le C \int_\bbr a' \Big( \left| \eta'(u) - \eta'(\bar u) \right| \left|\eta''(u) - \eta''(S) \right| + \left| \eta'(\bar u) - \eta'(S) \right| \left|\eta''(u) - \eta''(\bar u) \right|  \Big) d\xi\\
&\le C \int_\bbr a' \Big( \left| \eta'(u) - \eta'(\bar u) \right| \left( |\eta'(u) - \eta'(\bar u)| +\delta_1 \right) +\delta_1 \left|\eta'(u) - \eta'(\bar u) \right|  \Big) d\xi\\
&\le C \int_\bbr a' \Big( \left| \eta'(u) - \eta'(\bar u) \right|^2+ \left|\eta'(u) - \eta'(\bar u) \right|  \Big) d\xi  
\le C\sqrt{\frac{\eps}{\lambda}} \mathcal{D}(u).
\end{aligned}
\end{align*}
Therefore,
\beq\label{B6}
|\mathcal{B}_6(u)| \le C\sqrt{\frac{\eps}{\lambda}} \mathcal{D}(u) + C\delta_0\frac{\eps}{\lambda} \mathcal{G}_0(\bar u).
\eeq

As in the estimate of $|\mathcal{B}_5(u)|$, we first have
\begin{align*}
\begin{aligned}
|\mathcal{B}_7(u)| &\le \delta_0 \int_\bbr a \mu(u)\left| \partial_\xi  \left( \eta'(u) - \eta'(S) \right) \right|^2 d\xi +\frac{C}{\delta_0} \int_\bbr |S'|^2  \left| \mu(u) - \mu(S) \right|^2 d\xi \\
&\le \delta_0 \mathcal{D}(u) + \underbrace{C \eps  \int_\bbr a'  \left| \eta''(u) - \eta''(S) \right|^2 d\xi }_{=:B_{71}(u)}.
\end{aligned}
\end{align*}
Following the same estimates as before, we have
\begin{align*}
\begin{aligned}
|B_{71}(u)-B_{71}(\bar u)| &\le  \int_\bbr a'  \left|  \left| \eta''(u) - \eta''(S) \right|^2-  \left| \eta''(\bar u) - \eta''(S) \right|^2 \right| d\xi\\
&= C  \int_\bbr a'  \left| \eta''(u) - \eta''(\bar u) \right| \left|\eta''(u)+ \eta''(\bar u) - 2\eta''(S) \right| d\xi \\
 &\le C  \int_\bbr a'  \left| \eta'(u) - \eta'(\bar u) \right| \Big(\left|\eta'(u)- \eta'(\bar u) \right| + 2 \left| \eta'(\bar u)-\eta'(S) \right|\Big) d\xi \\
  &\le C  \int_\bbr a'  \Big(\left|\eta'(u)- \eta'(\bar u) \right|^2 + 2\delta_1 \left|\eta'(u)- \eta'(\bar u) \right| \Big) d\xi 
   \le  C\sqrt{\frac{\eps}{\lambda}} \mathcal{D}(u),
\end{aligned}
\end{align*}
and
\begin{align*}
\begin{aligned}
 |B_{71}(\bar u)| \le C \eps \int_\bbr a'  \left| \eta'(\bar u) - \eta'(S) \right|^2 d\xi \le C \eps \int_\bbr a'   \eta(\bar u|S) d\xi \le C\delta_0\frac{\eps}{\lambda} \mathcal{G}_0(\bar u).
\end{aligned}
\end{align*}
Therefore, 
\beq\label{B7}
|\mathcal{B}_7(u)| \le C\delta_0^{1/2}  \mathcal{D}(u) + C\delta_0\frac{\eps}{\lambda} \mathcal{G}_0(\bar u).
\eeq

We first separate $|\mathcal{B}_8(u)|$ into
\[
|\mathcal{B}_8(u)| \le |\mathcal{B}_8(u)-\mathcal{B}_8(\bar u)| +|\mathcal{B}_8(\bar u)|.
\]
Since $|\bar u-S|\le |\bar u|+|S|\le 5|u_-|$, we use Taylor theorem to have
\[
|(\eta')(\bar u|S)|\le C |\bar u-S|^2,
\]
which together with \eqref{sec-S} and \eqref{e-est0} yields
\[
|\mathcal{B}_8(\bar u)| \le \eps \int_\bbr a'  |(\eta')(\bar u|S)| d\xi  \le C \eps \int_\bbr a'   \eta(\bar u|S) d\xi \le C\delta_0\frac{\eps}{\lambda} \mathcal{G}_0(\bar u).
\]
Using \eqref{l3} and \eqref{l4}, we have
\begin{align*}
\begin{aligned}
|\mathcal{B}_8(u)-\mathcal{B}_8(\bar u)|  \le  \eps \int_\bbr a'  \left| \eta'(u) - \eta'(\bar u) \right| d\xi + C \eps \int_\bbr a'   |u-\bar u| d\xi \le C\sqrt{\frac{\eps}{\lambda}} \mathcal{D}(u).
\end{aligned}
\end{align*}
Therefore, 
\beq\label{B8}
|\mathcal{B}_8(u)| \le C\sqrt{\frac{\eps}{\lambda}} \mathcal{D}(u)+ C\delta_0\frac{\eps}{\lambda} \mathcal{G}_0(\bar u).
\eeq

Hence, we combine \eqref{B5}-\eqref{B8} to find that for $\eps/\lambda <\delta_0\ll1$,
\[
\sum_{i=5}^8 |\mathcal{B}_i(u)| \le \delta_0^{1/3}  \mathcal{D}(u) + C\delta_0\frac{\eps}{\lambda} \mathcal{G}_0(\bar u).
\]

$\bullet$ {\it Proof of \eqref{n3}} :
Since $|\bar u-S|\le |\bar u|+|S|\le 5|u_-|$, we use Taylor theorem to have
\[
|F(\bar u|S)| + \left| \eta'(\bar u) - \eta'(S) \right|\left| f(\bar u) - f(S) \right| + |f(\bar u|S)| + |\eta(\bar u|S)| \le C|\bar u -S|^2.
\]
Therefore, using \eqref{e-est0} and \eqref{l1} together with \eqref{eq_G}, we have
\[
\sum_{i=1}^4 |\mathcal{B}_i(\bar u)| \le C  \int_\bbr a' |\bar u -S|^2 d\xi \le C  \int_\bbr a' \eta(\bar u | S) d\xi \le C  \int_\bbr a' \eta(u | S) d\xi \le C\frac{\eps^2}{\lambda}. 
\]
Hence, using \eqref{bad0}, \eqref{n1}, \eqref{n2} and \eqref{l1} with \eqref{eq_G}, and taking $\delta_0$ small enough, there exists $C^*>0$ such that
\[
|\mathcal{B}(u)| \le \delta_0^{1/4} \mathcal{D}(u) + C\delta_0\frac{\eps}{\lambda} \mathcal{G}_0(\bar u) \le \delta_0^{1/4} \mathcal{D}(u) + C^* \frac{\eps^2}{\lambda}.
\]

$\bullet$ {\it Proof of \eqref{m1}} : 
Recall the functional $Y$ in \eqref{badgood}. Since
\[
|\eta(u|S)-\eta(\bar u|S)|\le |\eta'(u)-\eta'(\bar u)| + C|u-\bar u|,
\]
we use \eqref{l3} and \eqref{l4} to have
\beq\label{rel-Y}
|Y(u)-Y(\bar u)| \le \int_\bbr a' \Big( |\eta'(u)-\eta'(\bar u)| + C |u-\bar u| \Big) d\xi \le C\sqrt{\frac{\eps}{\lambda}} \mathcal{D}(u).
\eeq
Therefore, if $\mathcal{D}(u)\leq \frac{C^*}{4}\frac{\eps^2}{\lambda}$, then
\[
|Y(u)-Y(\bar u)| \le C \frac{\eps^2}{\lambda} \sqrt{\frac{\eps}{\lambda}}.
\]
Hence this and \eqref{rel-Y} yield
\[
|Y(u)-Y(\bar u)|^2 \le C \frac{\eps^2}{\lambda} \frac{\eps}{\lambda}\mathcal{D}(u).
\]

\vspace{0.5cm}

\subsection{Conclusion}\label{sub:con}
We split the proof into two steps, depending on the strength of the diffusion term $\mathcal{D}(u)$. 
\vskip0.2cm
\noindent{\it Step 1:}  
We first consider the case of
$
\mathcal{D}(u)\geq 4 C^* \frac{\eps^2}{\lambda}, $  where the constant $C^*$ is defined as in Proposition \ref{prop_out}.
Then using $\eqref{n3}$ and taking $\delta_0$ small enough, we have
\begin{align*}
\begin{aligned}
\mathcal{R}(u) &:= -\frac{1}{\eps^4}Y^2(u) +\mathcal{B}(u)+\delta_0 \frac{\eps}{\lambda} |\mathcal{B}(u)| - \mathcal{G}(u) \\
&\leq 2|\mathcal{B}(u)|- \mathcal{D}(u)\\
& \leq 2C^*\frac{\eps^2}{\lambda}-\left(1-2\delta_0^{1/4}
\right)\mathcal{D}(u)\\
& \leq 2 C^* \frac{\eps^2}{\lambda}-\frac{1}{2}\mathcal{D}(u)\leq 0,
\end{aligned}
\end{align*}
which gives the desired result.

\vskip0.2cm
\noindent{\it Step 2:}  
We now assume the other alternative, i.e., $\mathcal{D}(u)\leq 4 C^* \frac{\eps^2}{\lambda}.$ \\
First of all, we recall the the fixed small constant $\delta_1$ of Proposition \ref{prop:near} associated to the constant $K$, such that
\[
\mathcal{R}_{\eps,\delta_1}(\bar u)\le 0\quad \mbox{by }\eqref{final-R}.
\]
Since
\[
|Y(\bar u)|^2 \le 2( |Y(u)|^2 +|Y(u)-Y(\bar u)|^2 ),
\]
we have
\[
-2|Y(u)|^2 \le -|Y(\bar u)|^2 + 2|Y(u)-Y(\bar u)|^2.
\]
Then, recalling the functionals $\mathcal{B}=\sum_{i=1}^8 \mathcal{B}_i$ and $\mathcal{G}=\mathcal{G}_0 + \mathcal{D}$, we use \eqref{eq_G} and \eqref{eq_D} to find that for $\delta_0$ small enough, and for any $\delta_0^{-1}\eps<\lambda<\delta_0$,
\begin{align*}
\begin{aligned}
\mathcal{R}(u)&\le-\frac{2|Y(u)|^2}{\eps\delta_1}+ \mathcal{B}(u)+\delta_0\frac{\eps}{\lambda} |\mathcal{B}(u)|-\mathcal{G}(u)\\
&\leq -\frac{|Y(\bar u)|^2}{\eps\delta_1}+\sum_{i=1}^4 \mathcal{B}_i(\bar u) +\delta_0\frac{\eps}{\lambda} \left|\sum_{i=1}^4 \mathcal{B}_i(\bar u) \right|  -\left(1-\delta_1\frac{\eps}{\lambda}\right)\mathcal{G}_0(\bar u)-(1-\delta_1)\mathcal{D}(\bar u)\\
&\quad \underbrace{+\frac{2}{\eps\delta_1} |Y(u)-Y(\bar u)|^2 }_{=:J_1} + \underbrace{\left(1+\delta_0\frac{\eps}{\lambda}\right)\left(\sum_{i=1}^4|\mathcal{B}_i(u)-\mathcal{B}_i(\bar u)|+\sum_{i=5}^8|\mathcal{B}_i(u)| \right)}_{=:J_2} \\
&\quad-\delta_1\frac{\eps}{\lambda}\mathcal{G}_0(\bar u)  -\delta_1 \mathcal{D}( u).
\end{aligned}
\end{align*}
We claim that $J_1$ and $J_2$ are controlled by the last line above.  
Indeed, it follows from \eqref{n1}, \eqref{n2} and \eqref{m1} that for $\delta_0$ small enough with $\delta_0<\delta_1^4$, and  for any $\eps/\lambda<\delta_0$,
\begin{align*}
\begin{aligned}
J_1&\le \frac{C}{\delta_1}\Big(\frac{\eps}{\lambda} \Big)^2\mathcal{D}(u)\le  \frac{C\delta_0^2}{\delta_1}\mathcal{D}(u) \le  \frac{\delta_1}{2} \mathcal{D}(u),\\
J_2&\le  (C\delta_0^{1/2}+\delta_0^{1/3} ) \mathcal{D}(u) + C\delta_0\frac{\eps}{\lambda} \mathcal{G}_0(\bar u) \le \frac{\delta_1}{2} \mathcal{D}(u) +\delta_1 \frac{\eps}{\lambda} \mathcal{G}_0(\bar u).
\end{aligned}
\end{align*}
Hence we use \eqref{final-R} to have
\[
\mathcal{R}(u)\leq -\frac{|Y(\bar u)|^2}{\eps\delta_1}+\sum_{i=1}^4 \mathcal{B}_i(\bar u) +\delta_1\frac{\eps}{\lambda} \left|\sum_{i=1}^4 \mathcal{B}_i(\bar u) \right|  -\left(1-\delta_1\frac{\eps}{\lambda}\right)\mathcal{G}_0(\bar u)-(1-\delta_1)\mathcal{D}(\bar u) \le0,
\]
which completes the proof.

\begin{appendix}
\setcounter{equation}{0}

\section{Existence of entropies}\label{app-entropy}
We first provide a proof for the existence of entropies verifying the hypotheses (${\mathcal H}1$)-(${\mathcal H}2$).
\begin{lemma}
Let $f : \bbr\to\bbr$ be a smooth (or ${C}^3$) and strictly convex function satisfying the growth condition: there exist constants $a, b>0$ such that $|f(u)|\le  a e^{b|u|}$ for all $u\in\bbr$.
Then, there exists a function $\eta : \bbr\to\bbr$ such that $\eta$ is strictly convex and smooth satisfying the following hypotheses:
\begin{itemize}
\item  
(${\mathcal H}1$): There exists $\alpha>0$ such that $\eta''(u)\ge \alpha$ and $\eta''''(u)\ge \alpha$ for all $u\in\bbr$.\\

\item  
(${\mathcal H}2$): For a given constant $\theta>0$, there exists a constant $C>0$ such that 
for any $u,v\in \bbr$ with $|v|\le \theta$, the following inequalities holds:\\

(i) $ |\eta'(u)-\eta'(v)|^2 {\mathbf 1}_{\{ |u| \le 2\theta \}} + |\eta'(u)-\eta'(v)|{\mathbf 1}_{\{|u| >2\theta \}} \le C \eta(u|v)$\\

(ii) $|f(u)-f(v)|\le C |\eta'(u)-\eta'(v)|$\\

(iii) $|\eta''(u)-\eta''(v)|\le C |\eta'(u)-\eta'(v)|$\\

(iv) $\Big|\int^u_v \eta''(w)f(w)\, dw \Big| \le C \Big( |\eta'(u)-\eta'(v)| {\mathbf 1}_{\{ |u| \le 2\theta \}}  +|\eta'(u)-\eta'(v)|^2 {\mathbf 1}_{\{|u| >2\theta \}}\Big)$.
\end{itemize}
\end{lemma}

\begin{proof}
Let 
\[
\eta(u)=e^{b u} +e^{-bu} +u^4+u^2,
\]
where the positive constant $b$ is as in the growth condition of $f$.\\
Then we will show that $\eta$ verifies the all hypotheses.\\
First, for all $u\in\bbr$,
\[
\eta''(u)=b^2(e^{b u} +e^{-bu} ) + 12 u^2 +2 \ge 2,\quad \eta''''(u) = b^4(e^{b u} +e^{-bu} )  +24 \ge 24,
\]
which verifies (${\mathcal H}1$).\\
Below, the constant $C$ may change from line to line, but depend only on $a, b, \theta$.\\
For (i) of (${\mathcal H}2$), we first use the mean-value theorem together with \eqref{e-est0} to have
\[
 |\eta'(u)-\eta'(v)|^2 {\mathbf 1}_{\{ |u| \le 2\theta \}} \le C |u-v|^2 {\mathbf 1}_{\{ |u| \le 2\theta \}} \le C  \eta(u|v){\mathbf 1}_{\{ |u| \le 2\theta \}}.
\]
Since 
\[
\lim_{|u|\to\infty}\Big|\frac{\eta'(u)}{\eta(u)} \Big| =b >0,
\]
$ |\eta'(u)-\eta'(v)|{\mathbf 1}_{\{|u| >2\theta \}} \le C \eta(u|v){\mathbf 1}_{\{ |u| > 2\theta \}}$. More precisely, using the definition of $\eta$ and $\eta(\cdot|\cdot)$, we find that for  $|u| >2\theta$ and $|v|\le \theta$,
\[
 |\eta'(u)-\eta'(v)| \le  |\eta'(u)| + |\eta'(\theta)| \le  (b+C) e^{b|u|} \le C \eta(u|v).
\]
Therefore, we have (i) of (${\mathcal H}1$).\\
Likewise, we can get the remaining estimates (ii)-(iv) of  (${\mathcal H}1$).
Indeed, when $|u|\le 2\theta$, the both sides of (ii)-(iv) are all comparable with $|u-v|$, by the mean-value theorem. Also, we observe that for $|u|> 2\theta$, 
\[
|f(u)-f(v)|\sim |\eta'(u)-\eta'(v)|\sim |\eta''(u)-\eta''(v)|\sim e^{b|u|},
\]
and 
\[
\Big|\int^u_v \eta''(w)f(w)\, dw \Big| \sim |\eta'(u)-\eta'(v)|^2 \sim e^{2b|u|}.
\]
\end{proof}

The next lemma straightforwardly follows from the above lemma, since the growth condition of the flux is not necessary in a bounded interval.

\begin{lemma}
Let $f : \bbr\to\bbr$ be a smooth (or ${C}^3$) and strictly convex function.
Then, there exists a function $\eta : \bbr\to\bbr$ such that $\eta$ is strictly convex and smooth satisfying the following hypotheses:
\begin{itemize}
\item  
(${\mathcal H}1$): There exists $\alpha>0$ such that $\eta''(u)\ge \alpha$ and $\eta''''(u)\ge \alpha$ for all $u\in\bbr$.\\

\item  
(${\mathcal H}2'$): For a given constant $M>0$, there exists a constant $C>0$ such that 
for any $u,v\in \bbr$ satisfying $ |u| \le M $ and $|v|\le M/2$, the following inequalities holds:\\

(i) $ |\eta'(u)-\eta'(v)|^2  \le C \eta(u|v)$\\

(ii) $|f(u)-f(v)|\le C |\eta'(u)-\eta'(v)|$\\

(iii) $|\eta''(u)-\eta''(v)|\le C |\eta'(u)-\eta'(v)|$\\

(iv) $\Big|\int^u_v \eta''(w)f(w)\, dw \Big| \le C |\eta'(u)-\eta'(v)| $.
\end{itemize}
\end{lemma}
\begin{proof}
Let 
\[
\eta(u) = u^4+u^2.
\]
Then, $\eta$ satisfies (${\mathcal H}1$).\\
For any $u,v\in \bbr$ satisfying $ |u| \le M $ and $|v|\le M/2$, the estimates (i)-(iv) hold by the mean-value theorem.
\end{proof}

\end{appendix}

\bibliography{viscous_scalar_general_flux}
\end{document}